\documentclass[11pt]{amsart}

\usepackage{amssymb}
\usepackage{enumitem}
\usepackage[dvipsnames]{xcolor}
\usepackage[normalem]{ulem}

\usepackage{color}

\usepackage[a4paper,top=3 cm,bottom=3 cm,left=2.5 cm,right=2.5 cm]{geometry}

\newtheorem{theorem}{Theorem}[section]
\newtheorem{lemma}{Lemma}[section]
\newtheorem{prop}{Proposition}[section]
\newtheorem{coro}{Corollary}[section]
\newtheorem{remark}{Remark}[section]
\newtheorem{defn}{Definition}[section]

\numberwithin{equation}{section}

\def\to{\rightarrow}

\renewcommand{\div}{\mathrm{div} \hspace{0.5mm}}     

\newcommand{\nchi}{{\raise.3ex\hbox{$\chi$}}}

\def\XXint#1#2#3{{\setbox0=\hbox{$#1{#2#3}{\int}$ }
		\vcen{\hbox{$#2#3$ }}\kern-.6\wd0}}

\newcommand{\cC}{\mathcal{C}}
\newcommand{\cD}{\mathcal{D}}
\newcommand{\cE}{\mathcal{E}}

\newcommand{\cM}{\mathcal{M}}

\renewcommand{\det}{{\rm det}}

\newcommand{\N}{\mathbb{N}}

\newcommand{\R}{\mathbb{R}}

\newcommand{\T}{\mathbb{T}}

\makeatletter
\newcommand{\justified}{%
	\rightskip\z@skip%
	\leftskip\z@skip}
\makeatother

\newcommand\restr[2]{{
		\left.\kern-\nulldelimiterspace 
		#1 
		\vphantom{\big|} 
		\right|_{#2} 
}}

\DeclareFontFamily{U}{mathx}{\hyphenchar\font45}
\DeclareFontShape{U}{mathx}{m}{n}{<-> mathx10}{}
\DeclareSymbolFont{mathx}{U}{mathx}{m}{n}
\DeclareMathAccent{\widebar}{0}{mathx}{"73}

\renewcommand{\i}{\ifmmode\mathit{\mathchar"7010 }\else\char"10 \fi}
\renewcommand{\j}{\ifmmode\mathit{\mathchar"7011 }\else\char"11 \fi}

\def\char{{1\!\mbox{\rm l}}}

\newcommand{\with}{\quad\!\hbox{with}\!\quad}
\newcommand{\andf}{\quad\!\hbox{and}\!\quad}

\allowdisplaybreaks

\title[Vlasov-Navier-Stokes equations]{Large-time asymptotics of periodic two-dimensional Vlasov-Navier-Stokes flows}

\author{Rapha\"el Danchin and Ling-Yun Shou}
\date{}

\begin{document}

\begin{abstract}
We study the large-time behavior of finite-energy weak solutions for the Vlasov-Navier-Stokes equations in a two-dimensional torus. We focus first on the homogeneous case where the ambient (incompressible and viscous) fluid carrying the particles has a constant density, and then on the variable-density case. In both cases, large-time convergence to a monokinetic final state is demonstrated. 
For any finite energy initial data, we exhibit 
an algebraic convergence rate that deteriorates as the initial particle distribution increases. When the initial particle distribution is suitably small, then the convergence rate becomes exponential, {a result consistent with} the work of Han-Kwan et al. \cite{han1} dedicated to the homogeneous, three-dimensional case, where an additional smallness condition on the velocity was required. 
In the non-homogeneous case, we establish similar stability results, allowing a piecewise constant fluid density with jumps.
\end{abstract}


\maketitle

\section{Introduction}

In this paper, we investigate the large-time behavior of global solutions to two types of Vlasov-Navier-Stokes equations that describe the motion of particles dispersed in a viscous 
incompressible fluid \cite{o1,williams1}, contained in a two-dimensional torus $\mathbb{T}^2$. 

If the fluid is homogeneous, that is, with constant density, 
then the governing equations are:
\begin{equation}\label{homoVNS}
    \left\{
    \begin{aligned}
            &f_{t}+v\cdot\nabla_x f+ {\rm{div}}_v( (u-v)f)=0\quad&\hbox{in }\ \R_+\times\mathbb{T}^2\times \mathbb{R}^2,\\
        & u_{t}+ u\cdot\nabla u+\nabla P=\Delta u-\int_{\mathbb{R}^2}  (u-v)f\,dv\quad&\hbox{in }\ \R_+\times\mathbb{T}^2,\\
        &{\rm{div}}\, u=0\quad&\hbox{in }\ \R_+\times\mathbb{T}^2. 
    \end{aligned}
    \right.
\end{equation}
Above, $u=(u^1,u^2)(t,x)$ and $P=P(t,x)$ denote the velocity and the pressure of the fluid, respectively, at  time $t\in\mathbb{R}_{+}$ and position $x=(x_{1},x_{2})\in \mathbb{T}^{2}$, and $f=f(t,x,v)\geq0$ is the distribution function of particles moving with the 
kinetic velocity $v=(v_{1},v_{2})\in \mathbb{R}^{2}$. The system \eqref{homoVNS} couples a Vlasov equation with the incompressible Navier–Stokes equations through the Brinkman force 
$$
-\int_{\mathbb{R}^2}  (u-v)f\,dv.
$$
The macroscopic density, momentum and energy of the particles, denoted by  $n_{f}$, $j_{f}$ and $e_{f},$  respectively, are defined from  the distribution function 
$f$  as follows:
\begin{equation}\label{macro}    \begin{aligned}
    &n_f:=\int_{\mathbb{R}^2}f\, dv,\quad j_{f}:=\int_{\mathbb{R}^2}vf  dv\andf e_{f}:=\frac{1}{2}\int_{\mathbb{R}^2}|v|^2f\,dv.    \end{aligned}\end{equation}
    At the formal level, solutions of  \eqref{homoVNS}  supplemented with the initial data
\begin{equation}\label{homod}(f,u)|_{t=0}=(f_0,u_0)\end{equation}
        satisfy the following mass  and  momentum  conservation laws:
\begin{align}
\int_{\mathbb{T}^2}n_{f}(t)\,dx=\cM_0:=\int_{\mathbb{T}^2}n_{f_0}\,dx,\qquad \int_{\mathbb{T}^2}( u +j_{f})(t)\, dx=\int_{\mathbb{T}^2}( u_0 +j_{f_0})\, dx,\label{mass}
\end{align}
as well as  the  energy balance: 
\begin{equation}\label{energy}
\cE(t)+\int_{0}^{t}\cD(\tau)\,d\tau= \cE(0),
\end{equation}
where  the kinetic energy $\cE$ and the dissipation rate $\cD$ are defined by
\begin{align}
&\cE:=\frac{1}{2}\int_{\mathbb{T}^2} |u|^2\,dx+\frac{1}{2}\int_{\mathbb{T}^2\times\mathbb{R}^2}|v|^2f\, dxdv,\label{enrg}\\
&\cD:=\int_{\mathbb{T}^2}|\nabla u|^2\,dx+\int_{\mathbb{T}^2\times\mathbb{R}^2}|u-v|^2f\, dxdv.\label{enrgD}
\end{align}
For smooth solutions (with enough decay at infinity if applicable), the above relations 
are still valid if the fluid domain is the whole space $\R^d$ with $d=2,3$ or the three-dimensional torus $\T^3.$
Based on these relations,  one can prove the global existence of finite energy 
weak solutions in dimensions two and three (see \cite{Anoshchenko,boudin3,Yu1}). 
Furthermore,  as  first shown in  \cite{hankwanunique}, 
 uniqueness holds true in dimension two. All this can be summed up in the following statement: 
\begin{theorem}
Let  the fluid domain  $\Omega$ be the torus $\T^d$ or the whole space $\R^d$ with $d=2,3.$  Assume that 
$$f_0\in L^\infty(\Omega\times\R^d)\!\!\with\!\!  |v|^2f_0\in L^1(\Omega\times\R^d),\andf u_0\in L^2(\Omega)\!\!\with\!\! \div u_0=0.$$
Then, \eqref{homoVNS} supplemented with \eqref{homod} admits  a global-in-time distributional solution verifying 
$${\cE(t)\!+\!\int_0^t \cD(\tau)\,d\tau \leq \cE_{0}:= \cE(0)},
\quad  t\in\R_+.$$
If $d=2$ and, in addition, $|v|^qf_0\in L^\infty(\Omega\times\R^2)$ for some $q>4,$ then  uniqueness is true. 
\end{theorem}
 Understanding the large-time behavior of these weak solutions has been the subject of several recent papers. A point that makes this study particularly challenging is that the only nontrivial equilibria of the distribution function $f$ are singular in the sense that $f$ becomes {\emph{monokinetic}}: it is a Dirac measure with respect to the kinetic variable. Due to the Brinkman force, this causes some difficulties 
  when establishing uniform-in-time a priori estimates.

The large-time behavior of the solution strongly depends on the type of fluid domain that is considered. 
Typically, in the whole space situation, the velocity tends to $0$ with an algebraic {convergence} rate 
that is the same as for the heat equation (see \cite{danchinvns,HK3} for the $\R^3\times\R^3$ case), whereas for periodic boundary conditions and small solutions, the velocity  tends exponentially fast to
a constant state \cite{choiinhomo,han1}. 
The reader can also refer to  \cite{ErHKM1} for the bounded domains case and to \cite{Er2} for the half-space.

 Under the condition that the initial energy $\mathcal{E}_0$ and the $\dot{H}^{\frac{1}{2}}$-norm of the initial velocity are sufficiently small, Han-Kwan \emph{et al.} established in  \cite{han1} the uniform-in-time boundedness of $n_f,$ $ j_f,$  $e_f,$ and proved the exponential-in-time stability of global weak solutions  in three-dimensional periodic domains.
 In the two-dimensional case,  they  stated  exponential {convergence} estimates  for sufficiently small finite energy solutions. Very recently, Han-Kwan and Michel \cite{MR4813036} justified various hydrodynamic limits of the incompressible Vlasov–Navier–Stokes system in high friction regimes.

\medbreak
In some applications, it is relevant to take the variations of the density $\rho=\rho(t,x)\geq0$
of the ambient fluid.
Then, the motion of the particles and of the fluid is governed  by the 
following \emph{inhomogeneous} incompressible Vlasov-Navier-Stokes system
\begin{equation}\label{VNS}
    \left\{
    \begin{aligned}
            &f_{t}+v\cdot\nabla_x f+ {\rm{div}}_v( (u-v)f)=0&\hbox{in }\ \R_+\times\mathbb{T}^2\times \mathbb{R}^2,\\
        &\rho_{t}+{\rm{div}} (\rho u)=0&\hbox{in }\ \R_+\times\mathbb{T}^2,\\
        &(\rho u)_{t}+{\rm{div}} (\rho u\otimes u)+\nabla  P=\Delta u-\int_{\mathbb{R}^2}  (u-v)f\,dv&\hbox{in }\ \R_+\times\mathbb{T}^2,\\
        &{\rm{div}}\, u=0&\hbox{in }\ \R_+\times\mathbb{T}^2,\\        
    \end{aligned}
    \right.
\end{equation}
subject to the initial data
\begin{equation}\label{d}
\begin{aligned}
(f,\rho, u)|_{t=0}=(f_0,\rho_0, u_0).
\end{aligned}
\end{equation}
Yu and Wang \cite{wangd1} established the global existence of weak solutions to \eqref{VNS} with specular reflection boundary conditions. Choi \cite{choiinhomo} studied the long-time solvability of strong solutions in $\mathbb{T}^3\times\mathbb{R}^3$ or $\mathbb{R}^3\times\mathbb{R}^3$ and provided a conditional exponential {convergence} result in the periodic case. Li, Shou and Zhang \cite{LSZ1} studied the inhomogeneous Vlasov-Navier-Stokes equations in the presence of vacuum and proved the exponential stability in $\mathbb{R}^3\times\mathbb{R}^3$ when the initial energy is suitably small,  exhibiting  Lyapunov functionals 
and dissipation rates. The algebraic convergence of small solutions without vacuum in $\mathbb{R}^3\times\mathbb{R}^3$ has been established in \cite{SWYZ}.

\vspace{2mm}

 Our main goal is to get more accurate results for the large-time behavior of these solutions
\emph{without assuming that the initial energy is small}. 
 In fact, we shall show that general two-dimensional periodic weak solutions 
always have at least an algebraic {convergence} rate (depending on  the total mass of $f_0$) 
and that if $\|f_0\|_{L^\infty_{x,v}}$ is small enough, then the exponential convergence of solutions toward their equilibrium states holds true. 
In this latter case, we shall also specify the behaviors of the density function $n_f(t)$ of the particles and of the fluid density $\rho(t)$ when $t$ goes to infinity.
\medbreak\noindent{\bf Notation.} 
Throughout the paper $C$  denotes a `harmless' constant that may change from line to line, and 
$A\lesssim B$ means $A\leq CB.$ 
For functions $g$ depending on both $x\in\T^2$ and $v\in\R^2$, we shall sometimes use the 
following short notation for Lebesgue norms  (with $1\leq p,r\leq\infty$):
$$\|g\|_{L^p_{x,v}}:= \|g\|_{L^p(\T^2\times\R^2)}\andf \|g\|_{L^p_v(L^r_x)}:= \|g\|_{L^p(\R^2_v;(L^r(\T^2_x)))}.$$


\section{Main results}

We first present our results on the large-time behavior of the \emph{homogeneous} incompressible Vlasov-Navier-Stokes equations \eqref{homoVNS}, then their extension to the inhomogeneous setting.

\subsection{The homogeneous case}

  As first observed in  \cite{choiinhomo}
in the more general context of the inhomogeneous Vlasov-Navier-Stokes equations, 
in order to study the long-time asymptotics of solutions in the periodic case, it is wise
to use the following \emph{modulated energy} (or relative entropy) functional:
\begin{equation}\label{eq:modulated}
\mathcal{H}:=\frac{1}{2}\int_{{\mathbb{T}^2}}|u-\langle u\rangle |^2\,dx+\frac{1}{2}\int_{\mathbb{T}^2\times\mathbb{R}^2}\left|v-\frac{\langle j_{f}\rangle}{\langle n_{f}\rangle} \right|^2f\, dxdv+\frac{1}{2}\frac{\|n_{f}\|_{L^1}}{\langle n_{f}\rangle+1} \left|\langle u\rangle-\frac{\langle j_{f}\rangle}{\langle n_{f}\rangle}\right|^2,
\end{equation}
where $\langle \cdot\rangle:={\frac1{|\mathbb{T}^2|}}\int_{\mathbb{T}^2}\cdot\: dx$ denotes the average operator.

\medbreak
By \eqref{mass} and \eqref{energy}, one can deduce the  balance of modulated energy, namely
\begin{equation}\label{intenergym}
\mathcal{H}(t)+\int_0^t\cD(\tau)\,d\tau= \mathcal{H}_0:=\mathcal{H}(0).
\end{equation}
If $\mathcal H(t)$ converges to $0$ as $t\rightarrow\infty$, then 
Relations \eqref{mass} and the definition of $\mathcal H$  ensure that $u$ converges to the constant velocity
\begin{equation}
u_{\infty}:=\frac{\langle u_{0}+j_{f_{0}}\rangle}{1+\langle n_{f_0}\rangle}\cdotp\label{uinfty}
\end{equation}
Indeed, we observe that
\begin{equation}\label{uinftybis}
\langle u\rangle-u_\infty=\frac{\langle n_f\rangle}{1+\langle n_f\rangle}
\Biggl(\langle u\rangle-\frac{\langle j_f\rangle}{\langle n_f\rangle}\Biggr)\cdotp
\end{equation}
Our first result is the convergence to equilibrium for all finite energy solutions of \eqref{homoVNS}. 
\begin{theorem}\label{theorem1}
 Assume that $(f_0,u_0)$ satisfies \footnote{The last assumption is required  for uniqueness (see \cite{hankwanunique}). For existence, it suffices to suppose
 $|v|^2f_0\in L^1_{x,v}.$}
\begin{equation}\label{a1}
 u_0\in L^2, \quad\ \div u_0=0,\quad\ 
 0\leq f_0\in L^1_{x,v}\cap L^{\infty}_{x,v}\ \hbox{ and }\ \ |v|^{q}f_{0}\in  L^{\infty}_{x,v} \ \hbox{ for some }\  q>4.
\end{equation}
Then, Equations \eqref{homoVNS} supplemented with the initial data $(f_0,u_0)$ admit a unique global weak solution $(f,u,P)$ which satisfies for all $t\geq0,$
\begin{multline}\label{behavior1}
\mathcal{H}(t)+\frac{{\mathcal{M}_0}}{1+{\mathcal{M}_0}}\|u(t)-u_{\infty}\|_{L^2}^2+\| f(t)|v-u_{\infty}|^2\|_{L^1_{x,v}}\\
 \leq C\bigg( 1+\frac{\cM_0  t}{1+\mathcal{H}_0+\cM_0+\| f_{0}\log f_{0} \|_{L^1_{x,v}} } \bigg)^{-\frac{C}{\cM_0}}\mathcal{H}_0,
\end{multline}
where $C$ is  a positive constant depending only  on  $\T^2$.
\end{theorem}

Our second result states that  $u(t)$ tends to $u_\infty$ 
  exponentially fast when $t$ goes to infinity,  and specifies the large-time
behavior  of the distribution function $f$ 
under the additional condition that the $L^{\infty}_{x,v}$ norm of $f_0$ is sufficiently small.
In fact, due to the drag term, it is expected that $f(t,x,v)$ has exponential growth with respect to time
at some points of the phase space (see Formula \eqref{repr11}), while the total mass is conserved (see \eqref{mass}). 
Consequently, the limit distribution if it exists should be monokinetic and 
 concentrated at $v=u_\infty,$ namely of the form
 \begin{equation}\label{eq:limit} f(t,x,v)\rightarrow n_{\infty}(x-u_{\infty}t) \otimes \delta _{v=u_{\infty}}\quad\quad \text{as}\quad t\rightarrow \infty.\end{equation}
 The above relation reveals that if $u_{\infty}=0$, then the distribution function $f$ converges to a stationary solution while, when $u_{\infty}\neq 0$, the asymptotic behavior is that of a traveling wave. 
 As the limit is no longer a function,  a distance \emph{between measures} must be used to evaluate the speed of convergence.  Following \cite{han1}, we use the Wasserstein distance $W_1$ (see Definition \ref{MKD}). 
\smallbreak
Let us now state  our result:
\begin{theorem}\label{theorem11}
Assuming \eqref{a1},   there exists a constant $\alpha_1>0$ depending only on $q$, $\mathbb{T}^2$, $\mathcal H_0, \mathcal{M}_0$, $u_{\infty}$ and {{$\| (|v|^3+|v|^{q})f_0\|_{L^{\infty}_{x,v}}$}} such that if
\begin{equation}\label{smin1}
\|f_{0}\|_{L^{\infty}_{x,v}}\leq \alpha_1,
\end{equation}
then, for any $t\geq 1$, the global solution $(f,u,P)$ to System \eqref{homoVNS} obtained in Theorem \ref{theorem1} satisfies
\begin{multline}\label{behavior11}
\mathcal{H}(t)+
W_1(f(t),n_{\infty}(\cdot-u_{\infty}t) \otimes \delta _{v=u_{\infty}})+\| |v-u_{\infty}|^2 f(t)\|_{L^{1}_{x,v}}\\+\|n_f(t)-n_{\infty}{{(\cdot-u_{\infty}t)}}\|_{\dot{H}^{-1}}
+\|u(t)-u_{\infty}\|_{H^2}+\|{{\dot{u}}}(t)\|_{L^2}+\|\nabla P(t)\|_{L^2}\leq Ce^{-{\lambda_0} t},
\end{multline}
with  {{$\dot{u}=u+u\cdot\nabla u$}} and
\begin{equation}{\lambda_0}:= \frac{1}{C_1 \left(1+\cM_0\log{(1+\|{{|v-u_{\infty}|^3}}f_0\|_{L^{\infty}_{x,v}})}\right)},\label{lambda}
\end{equation}
where $C$ depends only on suitable norms of the initial data, $C_1$ depends only on 
$\T^2,$ $q$ and $\mathcal H_0$, and the profile $n_{\infty}\in \dot H^{-1}\cap L^\infty$ is defined by
\begin{equation}
n_{\infty}(x):= n_{f_0}-\div_{\!x}\int_0^{\infty}\!\!\int_{\mathbb{R}^2}(v-u_{\infty})f(\tau,x+\tau u_\infty,v)\,dvd\tau.\label{eq:ninfty}
\end{equation}
\end{theorem}
 In contrast to \cite{han1},  the smallness of $\mathcal{H}_0$ 
\emph{is not} required in Theorems \ref{theorem1} and \ref{theorem11}: 
 the initial velocity may be arbitrarily large. In fact, our analysis does not 
require  $\nabla u$ to be small  in $L^1(t_0,\infty;L^{\infty})$ for some (small) $t_0>0$ while it was a key ingredient in \cite{han1}. 

To obtain the first result, the main idea is to adapt the $L^1 \log L^1$-estimate of \cite{goudonhy1} to our situation. 
First, taking advantage of the Trudinger inequality will enable us  to get  the following coercivity inequality:
$$
\mathcal{H}(t)\lesssim \Big( 1+\mathcal{H}_0+\| n_f\log n_f \|_{L^1_{x,v}}+\cM_0  \Big)\cD(t).
$$
Then, the key observation is that the $L^1 \log L^1$ `norm' of $n_f$ has at most linear time growth. This property 
will allow us to get \eqref{behavior1} by integrating some suitable differential inequality.
 \smallbreak
With regard to the exponential convergence of the solution to equilibrium (Theorem \ref{theorem11}), our strategy is to construct coupled energy functional inequalities analogous to those in \cite{danchinvns,LSZ1} so as to control the higher–order derivatives of the fluid velocity. 
However, compared to  the 3D whole space case with small data treated in \cite{danchinvns,LSZ1}, one has
to face three additional difficulties:  
 we are considering \emph{large} initial velocities in a \emph{two}-dimensional 
\emph{periodic box}. The difficulty that is caused by the periodicity is encountered when 
handling the convective term $u\cdot\nabla u$: the average $\langle u\rangle$ of $u$ 
might have a linear time growth.
To overcome this, we include the convection term 
in our energy functionals; that is, we use the material derivative $\dot{u}$ rather than the time derivative 
 $u_t$ (see Lemma \ref{lemma32}).
  The difficulty caused by dimension $2$ is related to the fact that the space ${H}^1$ 
   is embedded in the ${\rm{BMO}}$ space rather than in the $L^\infty$ space. 
This difficulty is by-passed thanks to a cancellation property in the coupling between the pressure and the velocity, 
  combined with the celebrated $\div$-${\rm curl}$ result that is recalled in Lemma \ref{lemmaBMO}.
 Finally, the possible largeness of the initial velocity is handled by means of the following
 `time-splitting energy argument'. 
 Let us beforehand fix some  small constant $\eta>0$ and time  $t_\eta>1$ so large that
 the modulated energy $\mathcal{H}$ and the dissipation rate $\cD$ satisfy $\mathcal{H}(t_\eta)\leq\eta$ and $\cD(t_\eta)\leq\eta$
(note that $t_\eta$ can be bounded explicitly in terms of the data, due to our general algebraic convergence estimate \eqref{behavior1}).
Now, we argue as follows:
\begin{itemize}[leftmargin=8mm]
\item[---] In the interval $[0,t_\eta],$ we establish  \emph{time-dependent} upper bounds of $n_f$, $j_f$ and $e_f$, of the form $\mathcal{O}(1)+\mathcal{O}(1)e^{Ct_\eta}\|f_0\|_{L^{\infty}_{x,v}}$ (see Subsection \ref{subsectionlarge}).  
Therefore,  one can find some decreasing function $\alpha_1$ such that 
whenever $\|f_0\|_{L^{\infty}_{x,v}}\leq \alpha_1(t_\eta),$  the bounds  of $n_f$, $j_f$ and $e_f$  depend only on the  initial data, 
independently of  $\eta$.
\item[---] In the interval $[t_\eta,\infty),$ we perform  a bootstrap argument ensuring the Lipschitz bound 
\begin{equation}\label{eq:Lip0}\int_{t_\eta}^T\|\nabla u\|_{L^{\infty}}\,dt\leq \frac{1}{10}\quad\hbox{for all}\quad  T>t_\eta.\end{equation}
Taking advantage of  the change of the velocity variable $v$ originating from  \cite{han1} (and recalled in Lemma \ref{lemmaninfty}),
Inequality \eqref{eq:Lip0}  for some fixed $T>t_\eta$ allows us to derive \emph{time-independent} bounds of $n_f$, $j_f$ and $e_f$ on $[t_\eta,T].$ 
In order to get \eqref{eq:Lip0} \emph{for all $T$}, the key ingredient is to establish higher order energy estimates with exponential time weights.  
Following  \cite{danchinvns,LSZ1}, we employ three levels of energy functionals, 
which, combined with suitable functional embedding, provide us with a control of the  
$L^1(t_\eta,T;L^{\infty})$ norm of $\nabla u$ by some power of $\eta,$ whenever  $n_f$, $j_f$ and $e_f$ are under control. 
Leveraging a bootstrap argument eventually gives \eqref{eq:Lip0} with $T=\infty,$ if $\|f_0\|_{L^{\infty}_{x,v}}$ is small enough.
\end{itemize}


\subsection{The inhomogeneous case}

Our second goal is to extend Theorems \ref{theorem1} and \ref{theorem11} to the \emph{inhomogeneous} Vlasov-Navier-Stokes system \eqref{VNS} with, possibly,  vacuum and discontinuous fluid density. 

Smooth solutions of  \eqref{VNS} verify the following mass  and  momentum  conservation laws:
\begin{equation}\begin{aligned}
\int_{\mathbb{T}^2}\rho(t)\,dx=\int_{\mathbb{T}^2}\rho_0\,dx,&\qquad
\int_{\mathbb{T}^2}n_{f}(t)\,dx=\mathbf{M}_0:=\int_{\mathbb{T}^2}n_{f_0}\,dx,\\
\int_{\mathbb{T}^2}(\rho u +j_{f})(t)\, dx&=\int_{\mathbb{T}^2}( \rho_0 u_0 +j_{f_0})\, dx,\label{mass:in}
\end{aligned}\end{equation}
and the  energy balance: 
\begin{equation}\label{energyeq:in}
\mathbf{E}(t)+\int_{0}^{t}\mathbf{D}(\tau)\,d\tau= \mathbf{E}_0:=\mathbf{E}(0)
\end{equation}
with the kinetic energy
\begin{align}
&\mathbf{E}:=\frac{1}{2}\int_{\mathbb{T}^2} \rho |u|^2\,dx+\frac{1}{2}\int_{\mathbb{T}^2\times\mathbb{R}^2}|v|^2f\, dxdv,\label{enrg:in}
\end{align}
and the dissipation rate
\begin{align}
&\mathbf{D}:=\int_{\mathbb{T}^2}|\nabla u|^2\,dx+\int_{\mathbb{T}^2\times\mathbb{R}^2}|u-v|^2f\, dxdv.\label{enrgD:in}
\end{align}
 As in \cite{choiinhomo}, to investigate the convergence of $(f,\rho,u)$ to its equilibrium state, we introduce the {\emph{modulated energy}}
\begin{multline}\label{eq:modulated:in}
\mathbf{H}:=\frac{1}{2}\int_{\mathbb{T}^2}\rho\left|u-\frac{\langle \rho u\rangle}{\langle \rho\rangle} \right|^2\,dx+\frac{1}{2}\int_{\mathbb{T}^2\times\mathbb{R}^2}\left|v-\frac{\langle j_{f}\rangle}{\langle n_{f}\rangle} \right|^2f\, dxdv\\
+\frac{1}{2}\frac{\|n_f\|_{L^1}\|\rho\|_{L^1}}{\|n_f\|_{L^1}+\|\rho\|_{L^1}}
 \left|\frac{\langle \rho u\rangle}{\langle\rho\rangle}-\frac{\langle j_{f}\rangle}{\langle n_{f}\rangle}\right|^2.
\end{multline}
We still have the balance of modulated energy, namely
\begin{equation}\label{energym:in}
\mathbf{H}(t)+\int_0^t\mathbf{D}(\tau)\,d\tau=\mathbf{H}_0:=\mathbf{H}(0).
\end{equation}
When $\mathbf{H}(t)$ converges to $0$ as $t\rightarrow\infty$, one infers from \eqref{mass:in} and \eqref{eq:modulated:in} that  $u$ converges to the constant velocity field
\begin{equation}
\bar{u}_{\infty}:=\frac{\langle \rho_0 u_{0}+j_{f_{0}}\rangle}{\langle n_{f_0}\rangle+\langle \rho_0\rangle},   \label{uinfty:in}
\end{equation}
and that there exist two profiles $\bar{n}_{\infty}=\bar{n}_{\infty}(x)$ and $\bar{\rho}_{\infty}=\bar{\rho}_{\infty}(x)$ such that
$$
f(t,x,v)\rightarrow \bar{n}_{\infty}(x-\bar{u}_{\infty}t) \otimes \delta _{v=\bar u_{\infty}}\quad\text{and}\quad \rho(t,x)\rightarrow \bar{\rho}_{\infty}(x-\bar{u}_{\infty}t) \quad  \text{as}\quad t\rightarrow \infty.
$$

Our first result pertaining to \eqref{VNS} is stated as follows.
\begin{theorem}\label{theorem1:in}
 Assume that $(f_0,\rho_0,u_0)$ satisfies 
\begin{equation}\label{a1:in}
\begin{aligned}
& 0\leq \rho_0\in L^{\infty},\quad u_0\in L^2,\quad \div u_0=0,\quad  0\leq f_0\in L^1_{x,v}\cap L^{\infty}_{x,v}\andf |v|^{2}f_{0}\in  L^{1}_{x,v}.
 \end{aligned}
\end{equation}
Then, Equations \eqref{VNS} supplemented with the initial data $(f_0,\rho_0,u_0)$ admit a global weak solution $(f,\rho,u,P)$ satisfying for all $t\geq0,$
\begin{equation}\label{energy:in}
\mathbf{E}(t)+\int_{0}^{t}\mathbf{D}(\tau)\,d\tau\leq  \mathbf{E}_0,
\end{equation}
and, denoting $\mathbf{R_0}:=\|\rho_0\|_{L^\infty}(1+\mathbf{M}_0\|\rho_0\|_{L^1}^{-1})$, we have
\begin{multline}\label{behavior1:in}
\mathbf{H}(t)+\frac{\mathbf{M}_0}{\mathbf{M}_0+\|\rho_0\|_{L^1}}\|\sqrt{\rho}(t)(u(t)-\bar{u}_{\infty})\|_{L^2}^2+\| f(t)|v-\bar{u}_{\infty}|^2\|_{L^1_{x,v}}\\
 \leq C\bigg( 1+\frac{\mathbf{M}_0  t}{1+\mathbf{H}_0+\mathbf{R}_0+
 \| f_{0}\log f_{0} \|_{L^1_{x,v}}}  \bigg)^{-\frac{C}{\mathbf{M}_0}}\mathbf{H}_0,
\end{multline}
where $C$ is  a positive constant depending only  on  $\T^2$.
\end{theorem}

Finally, for initial density bounded away from zero and small enough $f_0,$ 
we establish exponential stability for \eqref{VNS} and get the first description of the large-time asymptotics of the fluid density $\rho$. In the  particular case $f_0=0$, we obtain exponential {convergence} estimates for the pure inhomogeneous Navier-Stokes flow, which complement those recently obtained in \cite{DW}. 
\begin{theorem}\label{theorem11:in}
Assume that \eqref{a1} holds, that $\rho_0$ is bounded away from zero, and that $|v|^{q}f_{0}\in  L^{\infty}_{x,v}$ for some $q>4$. There exists a constant $\alpha_2>0$ depending only on $q$, $\mathbb{T}^2$, $\mathbf H_0$, $\mathbf{M}_0$, $\bar{u}_{\infty}$, $\| (|v|^3+|v|^{q})f_0\|_{L^{\infty}_{x,v}}$, $\|\rho_0\|_{L^1}$ and $\|(\rho_0,\rho_0^{-1})\|_{L^{\infty}}$ such that, if
\begin{equation}\label{smin1:in}
\|f_{0}\|_{L^{\infty}_{x,v}}\leq \alpha_2,
\end{equation}
then, for any $t\geq 1$, the global solution $(f,u,P)$ to System \eqref{VNS} given by Theorem \ref{theorem1:in} satisfies
\begin{align}
\mathcal{H}(t)+
W_1&(f(t),\bar{n}_{\infty}(\cdot-\bar{u}_{\infty}t) \otimes \delta _{v=\bar{u}_{\infty}})+\| |v-\bar{u}_{\infty}|^2 f(t)\|_{L^{1}_{x,v}}\label{2.24}\\
&+\|n_f(t)-\bar{n}_{\infty}{{(\cdot-\bar u_{\infty}t)}}\|_{\dot{H}^{-1}}+\|\rho(t)-\bar{\rho}_{\infty}(\cdot-\bar{u}_{\infty} t)\|_{\dot{H}^{-1}} \nonumber\\
&\quad\quad\quad\quad+\|u(t)-\bar{u}_{\infty}\|_{H^2}+\|{{\dot{u}}}(t)\|_{L^2}+\|\nabla P(t)\|_{L^2}\leq Ce^{-{\lambda_1} t},\nonumber
\end{align}
with  {{$\dot{u}=u+u\cdot\nabla u$}} and
\begin{equation}
{\lambda_1}:= \frac{1}{C_2 \left(1+\mathbf{M}_0\log{(1+\||v-\bar u_{\infty}|^3f_0\|_{L^{\infty}_{x,v}})}\right)},\label{lambda1}
\end{equation}
where $C$ depends only on suitable norms of the initial data, $C_2$ depends only on 
$\T^2,$ $q$ and $\mathbf{H_0}$, and the profiles $\bar{n}_{\infty}, \bar{\rho}_\infty \in \dot H^{-1}\cap L^\infty$ are, respectively, defined by
\begin{equation}
\bar{n}_{\infty}{(x)}:= n_{f_0}{(x)}-\div_{\!x}\int_0^{\infty}\!\!\int_{\mathbb{R}^2}(v-\bar u_{\infty})f{(\tau,x+ \bar{u}_{\infty} \tau,v)}\,dvd\tau\label{eq:ninfty:in}
\end{equation}
and
\begin{equation}
\bar{\rho}_{\infty}(x):= \rho_0(x)-\div \int_0^{\infty} \bigl(\rho (u-\bar u_\infty)\bigr)(\tau,x+\bar u_\infty\tau)\, d\tau.\label{eq:rhoinfty}
\end{equation}
\end{theorem}

\begin{remark} 
If $(f_0,\rho_0,u_0)$ is sufficiently regular, then the  solutions obtained in Theorem \ref{theorem1:in} are unique.
 It would be of interest to investigate the uniqueness of these solutions if one only assumes that \eqref{a1:in} is satisfied.  
\end{remark}
\begin{remark}
In some applications, it is more relevant to assume that
the (density of the) drag force is $\rho(u-v) f$ (instead of just $(u-v)f$) 
so that the interaction between the particles and the surrounding fluid disappears in 
in vacuum.  More generally, if the drag force is of the form  
$\rho^{\sigma} (u-v)$ {\rm(}$\sigma\geq0${\rm)} then  the dissipation rate from the drag force satisfies
$$\int_{\mathbb{T}^2\times\mathbb{R}^2}\rho^{\sigma}|u-v|^2f\, dxdv\geq \inf_{x\in\mathbb{T}^2}\rho^{\sigma}(t,x) \int_{\mathbb{T}^2\times\mathbb{R}^2}|u-v|^2f\, dxdv.$$
Hence the algebraic decay result stated in Theorem \ref{theorem1:in} remains valid
whenever $\rho_0$ is bounded away from vacuum.\smallbreak
In this context, obtaining an exponential decay result analogous to Theorem \ref{theorem11:in} is considerably more subtle. 
It has been achieved by Choi and Kwon in \cite{choiinhomo} for smooth enough data and no vacuum, 
assuming an a priori bound on $n_f.$ In fact, 
it is not clear that one can have a good control 
of higher order energy functionals (see in particular \eqref{L1in} and \eqref{L2in}) without assuming 
that $\rho$ is Lispchitz. This Lipschitz control seems to be also needed when performing the 
change of variable along the characteristic in velocity (cf. Lemma \ref{lemmaninfty}). 
\end{remark}


\section{
The homogeneous case}\label{sectionhomo}

This section is devoted to the proof of Theorems \ref{theorem1} and \ref{theorem11}. Since the uniqueness of the two-dimensional weak solutions is guaranteed by \cite{hankwanunique}, we focus on the proof
of the existence of solutions that satisfy the desired {asymptotic} properties. 
Constructing these solutions  will be sketched in the last subsection.
For the time being, we focus on the proof of a priori estimates for smooth enough solutions, leading eventually to Inequalities \eqref{behavior1} and  \eqref{behavior11}.  

 Proving \eqref{behavior1} is carried out 
 in the first subsection. 
Next, as a preliminary step for getting   \eqref{behavior11},
 we establish three families of energy estimates for smooth solutions, that  already imply
 exponential {convergence}  of the solutions  {to equilibrium} emanating from initial data such that $\mathcal{H}_0$ and $f_0$ are small enough. 

 To handle the case with large $\mathcal{H}_0,$ the idea is as follows: for any small enough  $\eta,$  Inequality \eqref{behavior1} 
provides us with some  (explicit)  time $t_\eta>1$ such that
$\mathcal{H}(t_\eta)\leq \eta.$ By the same token, it is not difficult to ensure from \eqref{intenergym} that
$u(t_\eta)\in H^1$ and that $\cD(t_\eta)\leq\eta.$  On the time interval $[0,t_\eta]$, we establish  time-dependent upper bounds of $n_f$, $j_f$ and $e_f$. Once $t_\eta$ is fixed, we  take $\|f_0\|_{L^{\infty}_{x,v}}$ small enough  to have
 bounds of $n_f$, $j_f$ and $e_f$  on $[0,t_\eta]$ just in terms of the data.
 Then, the exponential  {convergence} estimates on $[t_\eta,T]$ will be achieved in five steps. The first step that crucially relies on Poincar\'e's inequality 
allows us to get exponential decay of the modulated energy~$\mathcal{H}$ provided $\|n_f\|_{L^\infty}$ is under control. 
This will be combined with our families of energy estimates  to prove 
exponential decay of  the dissipation rate $\cD$ and of $\|\dot u(t)\|_{L^2}$ whenever we have some 
control on   $\|n_f\|_{L^\infty}$ and  $\|e_f\|_{L^\infty}$ and (say):
\begin{equation}\label{eq:Lip}
\int_{t_\eta}^T \|\nabla u\|_{L^\infty}\,dt\leq\frac1{10}\cdotp\end{equation}
  Bounding the  $L^1(t_\eta,T;L^\infty)$ norm of $\nabla u$ in terms of the data
will be achieved in Step 4. 
Then, in the final step, we show that, under \eqref{eq:Lip},  one can bound   $\|n_f\|_{L^\infty},$
$\|j_f\|_{L^\infty}$ and  $\|e_f\|_{L^\infty}$ on $[t_\eta,T]$ in terms of $\|n_f(t_\eta)\|_{L^{\infty}}$, $\|j_f(t_\eta)\|_{L^{\infty}}$ and $\|e_f(t_\eta)\|_{L^{\infty}},$  and prove that \eqref{eq:Lip} is in fact a strict inequality, provided
$\eta$ and $\|f_0\|_{L^\infty_{x,v}}$ are small accordingly.  Then, a final bootstrap allows us to get  \eqref{eq:Lip}
with $T=\infty$, and then all the parts of \eqref{behavior11} pertaining to the velocity 
on $[t_\eta,\infty).$ Specifying the large-time behavior of $n_f$ is done independently, in the last-but-one subsection.

\subsection{Algebraic convergence rate for large data}\label{subsectionalgdecay}

We first show the algebraic decay of the modulated energy $\mathcal{H}$ of general (possibly large)
 weak solutions of System \eqref{homoVNS}.

\begin{prop}\label{propalge}
Let $(u,P,f)$ be a {smooth} global solution to System \eqref{homoVNS}. Then,
 there exists a  constant $c_1>0$ depending only on the fluid domain
 such that for any $t\geq0$ the modulated energy $\mathcal H$ defined in \eqref{eq:modulated} 
 satisfies
\begin{equation}\label{agdecay}
\begin{aligned}
&\mathcal{H}(t)\leq\Biggl( 1+\frac{\cM_0 t}{1+\mathcal{H}_0+\cM_0+\| f_{0}\log f_{0} \|_{L^1_{x,v}}  } \Biggr)^{-\frac{1}{c_1\cM_0}}\mathcal{H}_0.\end{aligned}
\end{equation}
\end{prop}

\begin{proof}
Poincar\'e's inequality gives us:
\begin{equation}\label{fluided}
\begin{aligned}
\int_{\mathbb{T}^2} |u-\langle u\rangle |^2\,dx\leq c_{\mathbb T^2}\|\nabla u\|_{L^2}^2,
\end{aligned}
\end{equation}
where $c_{\mathbb T^2}$ stands for the Poincar\'e  constant of the torus $\mathbb T^2.$ 
\medbreak
Next, to control the term of $\mathcal{H}$ corresponding to the energy of the particles, we make use of the dissipation provided 
by the Brinkman force. In fact, since 
$$\int _{\mathbb{T}^2\times\mathbb{R}^2}\biggl(v-\frac{\langle j_{f}\rangle}{\langle n_{f}\rangle}\biggr) f\, dxdv=0,$$
one has 
\begin{multline}\label{particleed111}
\int _{\mathbb{T}^2\times\mathbb{R}^2} |v-u|^2 f\,dxdv\\
\geq \frac{1}{2} \int _{\mathbb{T}^2\times\mathbb{R}^2} \left|v-\frac{\langle j_{f}\rangle}{\langle n_{f}\rangle}\right|^2 f\, dxdv
+\frac{\|n_f\|_{L^1}}{2}\left|\frac{\langle j_{f}\rangle}{\langle n_{f}\rangle}-\langle u\rangle\right|^2  -3 \int_{\mathbb{T}^2} |u-\langle u\rangle|^2 n_{f}\,dx.
\end{multline}
From  the definition of $\mathcal{H}$ in \eqref{eq:modulated}, \eqref{fluided} and \eqref{particleed111}, we infer that
\begin{equation}\label{particleed}
\begin{aligned}
\mathcal{H}\leq \int _{\mathbb{T}^2\times\mathbb{R}^2} |v-u|^2 f\,dxdv+\frac{c_{\mathbb{T}^2}}{2}\|\nabla u\|_{L^2}^2 +3 \int_{\mathbb{T}^2} |u-\langle u\rangle|^2 n_{f}\,dx.
\end{aligned}
\end{equation}
To estimate the last term, 
we argue as in  \cite{goudonhy1}: 
we write that 
\begin{equation}\label{212}
\begin{aligned}
\int_{\mathbb{T}^2} |u-\langle u\rangle|^2 n_{f}\,dx=\int_{\mathbb{T}^2} |u-\langle u\rangle|^2 n_f\,\mathbb{I}_{0\leq n_{f}\leq 1}\,dx+\int_{\mathbb{T}^2} |u-\langle u\rangle|^2 n_{f}\, \mathbb{I}_{n_{f}\geq 1}\,dx.
\end{aligned}
\end{equation}
The first term in the right-hand side of  \eqref{212} can be bounded as in \eqref{fluided}.
Estimating the last term relies on the following Trudinger inequality 
(see \cite[Page 162]{Trudinger}): 
\begin{equation}
\begin{aligned}
&\int_{\mathbb{T}^2}e^{  c|\Psi_{u}|^2}\,dx\leq K,
\end{aligned}
\end{equation}
where $\Psi_{u}=|u-\langle u\rangle |/\|\nabla u\|_{L^2}$ and $c,K$ are universal positive constants. 
\smallbreak
This guarantees that
\begin{align*}
\int_{\mathbb{T}^2} |u-\langle u&\rangle|^2 n_{f} \mathbb{I}_{n_{f}\geq 1}\,dx\\
&\quad =\|\nabla u\|_{L^2}^2\int_{\mathbb{T}^2} |\Psi_{u}|^2 n_{f} \mathbb{I}_{1\leq n_{f}\leq  e^{\frac{c}{2} |\Psi_{u}|^2} }\,dx+\|\nabla u\|_{L^2}^2\int_{\mathbb{T}^2} |\Psi_{u}|^2 n_{f} \mathbb{I}_{n_{f}\geq e^{\frac{c}{2} |\Psi_{u}|^2}}\,dx\\
&\quad\leq \|\nabla u\|_{L^2}^2\int_{\mathbb{T}^2} |\Psi_{u}|^2 e^{\frac{c}{2} |\Psi_{u}|^2}\,dx+\|\nabla u\|_{L^2}^2\int_{\mathbb{T}^2} |\Psi_{u}|^2 n_{f} \mathbb{I}_{|\Psi_{u}|^2\leq \frac{2}{c}\log n_{f}}\,dx\\
&\quad\leq \frac{2K}{c}\|\nabla u\|_{L^2}^2+\frac{2}{c}\|\nabla u\|_{L^2}^2\int_{\mathbb{T}^2}  n_{f}|\log n_{f}|\,dx.
\end{align*}
Thanks to Lemma \ref{lemmaflogf} with $\bar u=\langle j_f\rangle/\langle n_f\rangle$ and \eqref{eq:modulated}, the last term can be controlled as
follows:
\begin{equation}
\begin{aligned}
\int_{\mathbb{T}^2} n_{f}|\log n_{f}|\,dx&\leq \| f_{0}\log f_{0} \|_{L^1_{x,v}}+(\log(2\pi)+t)\cM_0+2e^{-1}|\mathbb{T}^2|+\mathcal{H}_0.\label{2141}
\end{aligned}
\end{equation}
By \eqref{212}-\eqref{2141} we thus obtain for some constant $c_1$ depending only 
on $\mathbb{T}^2,$
\begin{equation}\label{u2f}
\int_{\mathbb{T}^2} |u-\langle u\rangle|^2 n_{f}\,dx
\leq c_1\bigg(1+\mathcal{H}_0+\| f_{0}\log f_{0} \|_{L^1_{x,v}}+\cM_0  \biggr) \|\nabla u\|_{L^2}^2+2\cM_0 t\|\nabla u\|_{L^2}^2\cdotp
\end{equation}
Putting \eqref{particleed} and \eqref{u2f} together, we conclude that
$$\mathcal{H}\leq c_1\Big( 1+\mathcal{H}_0+\| f_{0}\log f_{0} \|_{L^1_{x,v}}+\cM_0  + \cM_0 t \Big)\cD.$$
Combining with the derivative of \eqref{intenergym}, we get
\begin{align*}
\frac{d}{dt}\mathcal{H}+\frac{1}{c_1\Big(1+\mathcal{H}_0+\| f_{0}\log f_{0} \|_{L^1_{x,v}}+\cM_0  +\cM_0 t \Big)}\mathcal{H}\leq 0.
\end{align*}
This implies that  for all $t\geq0,$
\begin{align*}
\mathcal{H}(t)&\leq {\rm exp}\bigg\{ -\int_{0}^{t} \frac{1}{c_1\Big(1+\mathcal{H}_0+\| f_{0}\log f_{0} \|_{L^1_{x,v}}+\cM_0  +\cM_0 \tau \Big)}\,d\tau \bigg\} \mathcal{H}_0\\
&={\rm exp}\bigg\{ -\frac{1}{c_1\cM_0 } \log\bigg( 1+\frac{\cM_0 t}{1+\mathcal{H}_0+\| f_{0}\log f_{0} \|_{L^1_{x,v}}+\cM_0} \bigg)\bigg\} \mathcal{H}_0,
\end{align*}
whence \eqref{agdecay}.
\end{proof}

Proposition \ref{propalge} implies the convergence of the solution  to its equilibrium state. Indeed, we have the following lemma.

\begin{lemma}\label{lemma51}
Let $(u,P,f)$ be a smooth global solution to System \eqref{homoVNS}. Then for any $t\geq0,$ it holds that
for some $C>0$ depending only on $\T^2,$
\begin{multline}\label{agdecay1}
\frac{{\mathcal{M}_0}}{1+{\mathcal{M}_0}}\|u(t)-u_{\infty}\|_{L^2}^2+\| |v-u_{\infty}|^2{f(t)}\|_{L^1_{x,v}}
\\ \leq C\Bigg( 1+\frac{\cM_0 t}{1+\mathcal{H}_0+\| f_{0}\log f_{0} \|_{L^1_{x,v}}+\cM_0 } \Bigg)^{-\frac{C}{\cM_0}} \mathcal{H}_0.
\end{multline}
\end{lemma}

\begin{proof}
The momentum conservation  \eqref{mass} implies  that
\begin{equation}\nonumber
\begin{aligned}
\langle u\rangle=\langle u_0+j_{f_0}\rangle-\langle j_{f}\rangle.
\end{aligned}
\end{equation}
Recall that $u_{\infty}$ is given by \eqref{uinfty}. Together with the mass conservation of $f$, we have
\begin{equation}\nonumber
\begin{aligned}
&\left| \frac{\langle j_f\rangle}{\langle n_f\rangle}-u_{\infty}\right|=\frac{1 }{1+\langle n_{f}\rangle }\left| \frac{\langle j_{f}\rangle}{\langle n_{f}\rangle} -\langle u\rangle\right|\leq \sqrt{\frac{2\mathcal{H}}{\cM_0 (1+\langle n_{f_0}\rangle)}  }\cdotp
\end{aligned}
\end{equation}
This leads to 
\begin{align*}
\int_{\mathbb{T}^2}\left| u-u_{\infty}\right|^2\,dx& \leq 3 \int_{\mathbb{T}^2}\Biggl(\left| u-\langle u\rangle\right|^2+\Biggl| \langle u\rangle-\frac{\langle j_f\rangle}{\langle n_f\rangle}\Biggr|^2+\Biggl|\frac{\langle j_f\rangle}{\langle n_f\rangle}-u_{\infty}\Biggr|^2\Biggr)\,dx \\
&\leq C\Biggl(1+\frac{1}{\langle n_{f_0}\rangle }\Biggr)\mathcal{H}
\end{align*}
and
\begin{equation}\nonumber
\begin{aligned}
\int_{\mathbb R^2\times\mathbb{T}^2} \left| v-u_{\infty}\right|^2 f\,dxdv&\leq 2\int_{\mathbb R^2\times\mathbb{T}^2} \left| v-\frac{\langle j_f\rangle}{\langle n_f\rangle}\right|^2 f\,dxdv + 2\langle n_{f_0}\rangle \left| \frac{\langle j_f\rangle}{\langle n_f\rangle}-u_{\infty} \right|^2 \leq C\mathcal{H}.
\end{aligned}
\end{equation}
Combining these estimates with \eqref{agdecay}, we arrive at \eqref{agdecay1}.
\end{proof}


\subsection{Higher order Lyapunov functionals}
Following \cite{LSZ1,danchinvns} where a similar approach was used in another context, here we
establish some inequalities involving higher-order norms of $u$. As already explained before, compared to the
aforementioned works, the key novelty is to replace $u_t$ by the material derivative 
$\dot{u}:=u_t+u\cdot\nabla u$ in the inequalities that are stated below:
\begin{lemma}\label{lemma32}
Let $(f,u,P)$ be a smooth solution  to System \eqref{homoVNS}. Let the dissipation rate $\mathcal D$ be  defined in \eqref{enrgD}.
There exists a universal constant $C\geq1$ such that the following holds:
\begin{itemize}
    \item $H^1$-estimates{\rm:} 
\begin{multline}\label{dpes021110000}
 \frac{d}{dt} \|\nabla u\|_{L^2}^2+\|\dot{u}\|_{L^2}^2+\frac{1}{C}(\|\nabla^2u\|_{L^2}^2+\|\nabla P\|_{L^2}^2)
\leq C\|\nabla u\|_{L^2}^4+C\|n_{f}\|_{L^{\infty}} \||u-v|^2f\|_{L^1_{x,v}}
    \end{multline}
    and
\begin{multline}\label{dpes0211100001}
\frac{d}{dt} \cD+\|\dot{u}\|_{L^2}^2+\frac{1}{C}(\|\nabla^2u\|_{L^2}^2+\|\nabla P\|_{L^2}^2)+ \||u-v|^2f\|_{L^1_{x,v}}\\
\leq C\|\nabla u\|_{L^2}^4+C(\|n_{f}\|_{L^{\infty}}+\|\nabla u\|_{L^{\infty}})  \||u-v|^2f\|_{L^1_{x,v}}.
    \end{multline}
        \item {Material derivative estimate\footnote{For any $2\times2$ matrices 
        $A$ and $B,$ we set $A:B=\sum_{i,j}A_{ij}B_{ij}$.}    
        \rm:}
\begin{multline}\label{3111}
 \frac{d}{dt}\int_{\mathbb{T}^2} \Big(|\dot{u}|^2-(P-\langle P\rangle)\nabla u: (\nabla u)^{\top} \Big)\,dx+ \|\nabla \dot{u}\|_{L^2}^2+  \|\sqrt{n_f}\dot{u}\|_{L^2}^2\leq  C\|\nabla u\|_{L^2}^2 \|\dot{u}\|_{L^2}^2\\
+C\Bigl(1+\|e_f\|_{L^{\infty}}+\|n_{f}\|_{L^{\infty}}\bigl(1+ \|u\|_{L^{\infty}}^2\bigr)
+\bigl(\|n_{f}\|_{L^{\infty}}+\|n_{f}\|_{L^{\infty}}^2\bigr)\cD\Bigr) \||u-v|^2f\|_{L^1_{x,v}}.
\end{multline}
\end{itemize}
\end{lemma}

\begin{proof}
Taking the inner product of $\eqref{homoVNS}_{2}$ with $\dot{u}$ yields
\begin{equation}\label{ddtnablauhomo}
\begin{aligned}
\int_{\mathbb{T}^2}  |\dot{u}|^2\, dx=\int_{\mathbb{T}^2}( \Delta u\cdot \dot{u}-\nabla P\cdot \dot{u})\,dx- \int_{\mathbb{T}^2\times\mathbb{R}^2} \dot{u} \cdot(u-v)f\,dxdv.
\end{aligned}
\end{equation}
One observes that
\begin{align*}
 \int_{\mathbb{T}^2}\Delta u\cdot \dot{u}\, dx&= \int_{\mathbb{T}^2}\Delta u\cdot u_t\, dx+\int_{\mathbb{T}^2} \Delta u\cdot (u\cdot \nabla u)\, dx\\
&=- \frac{1}{2}\frac{d}{dt}\|\nabla u\|_{L^2}^2-\sum_{i,j,k=1,2}\int_{\mathbb{T}^2} \partial_{i}u^j \partial_k u^i\partial_{k} u^j\,dx.
\end{align*}
The Gagliardo-Nirenberg inequality gives directly
 \begin{equation}\label{ddtnablauhomo2}
\sum_{i,j,k=1,2}\int_{\mathbb{T}^2} \partial_{i}u^j \partial_k u^i\partial_{k} u^j\,dx\leq \|\nabla u\|_{L^3}^3\leq C\|\nabla u\|_{L^2}^2 \|\nabla^2 u\|_{L^2}.\end{equation}
Next, it holds that
 \begin{equation}\label{eq:nablaP}
-\int_{\mathbb{T}^2}\nabla P\cdot \dot{u}\,dx=\sum_{i=1,2}\int_{\mathbb{T}^2} 
(P-\langle P\rangle) \partial_{i} u\cdot \nabla u^{i}\, dx\leq \|P-\langle P\rangle\|_{{\rm BMO}}\sum_{i=1,2}\|\partial_{i} u\cdot \nabla u^{i}\|_{\mathcal{H}^1},
\end{equation}
where we used $\div u=0$ and the duality between the Hardy space $\mathcal{H}^1 $ and the ${\rm BMO} $ space. 
Leveraging  $\div \partial_{i} u=0,$   $ \nabla\times \nabla u^{i}=0 $ and Lemma \ref{lemmaBMO}, we end up with
 \begin{equation}\label{ddtnablauhomo3}
-\int_{\mathbb{T}^2}\nabla P\cdot \dot{u}\,dx\leq C\|\nabla P\|_{L^2}\|\nabla u\|_{L^2}^2.
\end{equation}
From Young’s inequality, we directly deduce that
\begin{equation}\label{ddtnablauhomo4}
\begin{aligned}
\left| \int_{\mathbb{T}^2\times\mathbb{R}^2} \dot{u} \cdot(u-v)f\,dxdv\right|&\leq \frac{1}{4}\|\dot{u}\|_{L^2}^2+\bigg\|\int_{\mathbb{R}^2}(u-v)f\, dv\bigg\|_{L^2}^2\cdotp
\end{aligned}
\end{equation}
Fix some positive parameter $\beta_0$ to be chosen later. Then, collecting 
\eqref{ddtnablauhomo}--\eqref{ddtnablauhomo4} and using Young's inequality, we obtain
 \begin{equation}\label{32}
\frac{d}{dt} \|\nabla u\|_{L^2}^2+\frac{3}{2}\|\dot{u}\|_{L^2}^2\\
\leq \beta_0\|\nabla^2 u\|_{L^2}^2+C_{\beta_0}\biggl(\|\nabla u\|_{L^2}^4+\|\nabla P\|_{L^2}^2+\bigg\|\int_{\mathbb{R}^2}(u-v)f\, dv\bigg\|_{L^2}^2\biggr)\cdotp
\end{equation}
To bound $\nabla^2 u$ and $\nabla P,$ we rewrite  $\eqref{homoVNS}_2$ as the following Stokes system:
 \begin{equation}\label{eq:stokes}
  -\Delta u+\nabla P=-\dot{u}-\int_{\mathbb{R}^2}  (u-v)f\,dv,
  \qquad\div u=0.
 \end{equation}
Note that Lemma \ref{lemmastokes} implies that
 \begin{equation*}
 \|\nabla^2 u\|_{L^2}^2+\|\nabla P\|_{L^2}^2\leq  C\|\dot{u}\|_{L^2}^2+C\bigg\|\int_{\mathbb{R}^2}(u-v)f\, dv\bigg\|_{L^2}^2
 \end{equation*}
 and, owing to the Cauchy-Schwarz inequality and to the definition of $n_f$,  we have
 \begin{equation}\label{nabladragL2}
 \bigg\|\int_{\mathbb{R}^2}(u-v)f\, dv\bigg\|_{L^2}^2\leq \|n_f\|_{L^{\infty}} \int_{\mathbb{T}^2\times\mathbb{R}^2} |u-v|^2f\, dxdv.
\end{equation}
Hence 
 \begin{equation}\label{nablaunablaPL21}
 \|\nabla^2 u\|_{L^2}^2+\|\nabla P\|_{L^2}^2\leq C\|\dot{u}\|_{L^2}^2+
 C\|n_f\|_{L^{\infty}} \int_{\mathbb{T}^2\times\mathbb{R}^2} |u-v|^2f\, dxdv.
 \end{equation}
 Combining with \eqref{32} and choosing $\beta_0$ small enough yields the desired inequality \eqref{dpes021110000}.
\medbreak
To get \eqref{dpes0211100001}, the only difference compared to \eqref{dpes021110000} is the treatment of the last term 
in  \eqref{ddtnablauhomo}. It follows from $\eqref{homoVNS}_{1}$ that
\begin{align*}
- \int_{\mathbb{T}^2\times\mathbb{R}^2} \dot{u} \cdot(u-v)f\,dxdv
=-\frac{ 1}{2}\frac{d}{dt}& \int_{\mathbb{T}^2\times\mathbb{R}^2} (|u|^2-2u\cdot v) f\, dxdv+\frac{1}{2}\int_{\mathbb{T}^2\times\mathbb{R}^2} (|u|^2-2u\cdot v)f_{t}\,dxdv\\
&-  \int_{\mathbb{T}^2\times\mathbb{R}^2}  (u\cdot \nabla u)\cdot (u-v)f\, dxdv\\
=-\frac{1}{2}\frac{d}{dt}& \!\int_{\mathbb{T}^2\times\mathbb{R}^2}\! (|u|^2-2u\cdot v) f\, dxdv+ \!\int_{\mathbb{T}^2\times\mathbb{R}^2} (v\cdot\nabla u)\cdot (u-v)f\, dxdv\\
&-\int_{\mathbb{T}^2\times\mathbb{R}^2} u\cdot(u-v)f\, dxdv-  \int_{\mathbb{T}^2\times\mathbb{R}^2}  (u\cdot \nabla u)\cdot (u-v)f\, dxdv\\
=-\frac{1}{2}\frac{d}{dt}& \int_{\mathbb{T}^2\times\mathbb{R}^2} |u-v|^2f\, dxdv-\int_{\mathbb{T}^2\times\mathbb{R}^2} |u-v|^2f\, dxdv\\
&+\int_{\mathbb{T}^2\times\mathbb{R}^2} ((v-u)\cdot\nabla u)\cdot (u-v)f\, dxdv,
\end{align*}
where we used that the energy of the Vlasov equation $\eqref{homoVNS}_{1}$ satisfies
\begin{equation}
\begin{aligned}
&\frac{1}{2}\frac{d}{dt}\int_{\mathbb{T}^2\times\mathbb{R}^2} |v|^2 f\,dxdv= \int_{\mathbb{T}^2\times\mathbb{R}^2} v\cdot(u-v)f\, dxdv.\nonumber
\end{aligned}\end{equation}
Note that we have
$$\int_{\mathbb{T}^2\times\mathbb{R}^2} ((v-u)\cdot\nabla u)\cdot (u-v)f\, dxdv\leq \|\nabla u\|_{L^\infty}
\int_{\mathbb{T}^2\times\mathbb{R}^2}|u-v|^2f\,dxdv,$$
which allows to get \eqref{dpes0211100001}.
\medbreak
To prove Inequality \eqref{3111}, let us apply  $\partial_t+u\cdot \nabla $ to $\eqref{homoVNS}_2^{j}$ with $j=1,2.$ Then, using the equation $\eqref{homoVNS}_{1}$ yields
\begin{multline}\label{rhorutt}
 \partial_{t}\dot{u}^{j}- \Delta \dot{u}^{j} + n_{f} \dot{u}^{j}=-  
\partial_{i}( (\partial_{i}u\cdot\nabla) \nabla u^{j})- \div (\partial_{i}u \partial_{i}u^{j})\\
-\partial_{j}\partial_{t}P-(u\cdot\nabla )\partial_{j}P-\int_{\mathbb{R}^2}(u^j-v^j)(u-v)\cdot \nabla_{x}f\, dv+\int_{\mathbb{R}^2}(u^j-v^j)f\,dv.\end{multline}
Multiplying \eqref{rhorutt} by $\dot{u}^{j}$ and summing over $j=1,2$, we deduce after integration by parts that
\begin{equation}\label{311}
\begin{aligned}
&\frac{1}{2}\frac{d}{dt}\int_{\mathbb{T}^2} |\dot{u}|^2\,dx+ \int_{\mathbb{T}^2}|\nabla \dot{u}|^2\,dx+ \int_{\mathbb{T}^2} n_f|\dot{u}|^2\,dx=\sum_{i=1}^{5}I_{i},
\end{aligned}
\end{equation}
with
\begin{equation}\nonumber
\begin{aligned} 
&I_{1}=-  \sum_{i,j=1,2}\int_{\mathbb{T}^2}(\partial_{i}(\partial_{i}u\cdot \nabla u^{j})+\div (\partial_{i}u \partial_{i}u^{j}))\dot{u}^j \, dx,\\
&I_{2}=\int_{\mathbb{T}^2\times\mathbb{R}^2}(u-v)f\cdot \dot{u}\, dxdv,\\
&I_{3}=\int_{\mathbb{T}^2\times\mathbb{R}^2}(u-v)f \cdot (\nabla u\cdot \dot{u})\, dxdv,\\
&I_4=\int_{\mathbb{T}^2\times\mathbb{R}^2}(u-v)f\cdot (\nabla \dot{u}\cdot (u-v)) \,dxdv,\\
&I_{5}=-\sum_{j=1,2}\int_{\mathbb{T}^2}(\partial_{j}\partial_{t}P+(u\cdot\nabla )\partial_{j}P) \dot{u}^j\, dx.
\end{aligned}
\end{equation}
Integrating by parts, then taking advantage of  Gagliardo-Nirenberg and Young inequalities yields
\begin{equation*}
I_{1}\leq 2 \|\nabla u\|_{L^4}^2 \|\nabla \dot{u}\|_{L^2}\leq C\|\nabla u\|_{L^2}\|\nabla^2u\|_{L^2} \|\nabla \dot{u}\|_{L^2}\leq \frac{ 1}{4}\|\nabla \dot{u}\|_{L^2}^2+C\|\nabla u\|_{L^2}^2 \|\nabla^2u\|_{L^2}^2.
\end{equation*}
Hence, Inequality \eqref{nablaunablaPL21} allows to get
\begin{equation}\label{eq:I1} 
I_{1} \leq \frac{ 1}{4}\|\nabla \dot{u}\|_{L^2}^2+C\|\nabla u\|_{L^2}^2 \|\dot u\|_{L^2}^2
+C\|n_f\|_{L^\infty}\|\nabla u\|_{L^2}^2\int_{\mathbb{T}^2\times\mathbb{R}^2}|u-v|^2f\,dxdv.
\end{equation}
Keeping in mind  $\|\sqrt{n_{f}}\dot{u}\|_{L^2}^2$ in the left-hand side of \eqref{311}, we bound $I_2$ as follows:
$$I_{2}\leq \frac{1}{2}\|\sqrt{n_{f}}\dot{u}\|_{L^2}^2+ \frac12\int_{\mathbb{T}^2\times\mathbb{R}^2}|u-v|^2f\, dxdv.$$
Next, due to H\"older, Gagliardo-Nirenberg  and Sobolev inequalities, and to \eqref{nabladragL2}, $I_{3}$ can be controlled by
 \begin{equation*}
\begin{aligned} 
I_{3}&\leq \Bigl(\|\dot{u}-\langle \dot u\rangle\|_{L^4} \|\nabla u\|_{L^4}+|\langle \dot u\rangle|\|\nabla u\|_{L^2}\Bigr)
\left\|\int_{\mathbb{R}^2}(u-v)f\,dv\right\|_{L^2}\\
&\leq C\Bigl(\|\nabla \dot{u}\|_{L^2} \|\nabla u\|_{L^4}
+\|\dot u\|_{L^2}\|\nabla u\|_{L^2}\Bigr)
\left\|\int_{\mathbb{R}^2}(u-v)f\,dv\right\|_{L^2}\\
&\leq C\Bigl(\|\nabla \dot{u}\|_{L^2} \|\nabla u\|_{L^2}^{\frac{1}{2}}\|\nabla^2u\|_{L^2}^{\frac{1}{2}}+\|\dot u\|_{L^2}\|\nabla u\|_{L^2}\Bigr)
\|n_{f}\|_{L^{\infty}}^{\frac{1}{2}}\left(\int_{\mathbb{T}^2\times\mathbb{R}^2}|u-v|^2f\,dxdv\right)^{\frac{1}{2}}\\
&\leq \frac{1}{8} \|\nabla \dot{u}\|_{L^2}^2+C\|\nabla u\|_{L^2}^2\|\nabla^2 u\|_{L^2}^2
+C\|\dot u\|_{L^2}^2\|\nabla u\|_{L^2}^2
+C\|n_{f}\|_{L^{\infty}}\int_{\mathbb{T}^2\times\mathbb{R}^2}|u-v|^2f\,dxdv
\\&\hspace{8cm}+C\|n_{f}\|_{L^{\infty}}^2\left(\int_{\mathbb{T}^2\times\mathbb{R}^2}|u-v|^2f\,dxdv\right)^{2}\cdotp
\end{aligned}
\end{equation*}
Together with \eqref{nablaunablaPL21}, this yields
\begin{multline*}
I_{3}\leq  \frac{1}{8} \|\nabla \dot{u}\|_{L^2}^2+C\|\nabla u\|_{L^2}^2 \|\dot{u}\|_{L^2}^2
\\+C\|n_{f}\|_{L^{\infty}}\int_{\mathbb{T}^2\times\mathbb{R}^2}|u-v|^2f\,dxdv+C( \|n_{f}\|_{L^{\infty}}+\|n_{f}\|_{L^{\infty}}^2)\cD\int_{\mathbb{T}^2\times\mathbb{R}^2}|u-v|^2f\,dxdv.
\end{multline*}
To bound $I_4,$ it suffices to observe that
\begin{align*} 
I_4&\leq  \int_{\mathbb{T}^2} |\nabla\dot{u}| \bigg(\int_{\mathbb{R}^2} |u-v|^2f\,dv\biggr)^{\frac{1}{2}} 
\biggl( |n_{f}|^{\frac{1}{2}} |u|+|e_f|^{\frac{1}{2}}\biggr)\,dx\\
&\leq \frac{1}{8}\|\nabla \dot{u}\|_{L^2}^2+C(\|n_f\|_{L^{\infty}}\|u\|_{L^{\infty}}^2+\|e_f\|_{L^{\infty}})\int_{\mathbb{T}^2\times\mathbb{R}^2} |u-v|^2f\,dxdv.\end{align*}
The term $I_5$ requires a more complex computation. By $\div \partial_{i}u=0$, one has
\begin{equation}\nonumber
\begin{aligned} 
I_{5}&=\int_{\mathbb{T}^2} (P_t+u\cdot\nabla P) \div \dot{u}\,dx+\sum_{j=1,2}\int_{\mathbb{T}^2}\partial_{j}u^i \partial_{i} P\: \dot{u}^j\,dx\\
&=\sum_{j=1,2} \int_{\mathbb{T}^2} P_t \partial_{j}u^i \partial_{i}u^j\,dx+\sum_{i,j,k=1,2}\int_{\mathbb{T}^2}(\partial_{j}u^i \partial_{i} P \dot{u}^j-(P-\langle P\rangle) u^k \partial_{k}(\partial_{j}u^i \partial_{i}u^j))\,dx\\
&=\frac{d}{dt} \sum_{i,j=1,2} \int_{\mathbb{T}^2}( P-\langle P\rangle)  \partial_{j}u^i \partial_{i}u^j\,dx-\sum_{i,j=1,2}\int_{\mathbb{T}^2}(P-\langle P\rangle) \partial_{t}(\partial_{j} u^i \partial_{i}u^j)\,dx\\
&\qquad\qquad\qquad +\sum_{i,j,k=1,2}\int_{\mathbb{T}^2}\Big(\partial_{j}u^i \partial_{i} P \:\dot{u}^j-(P-\langle P\rangle)( u^k \partial_{k}\partial_{j}u^i \partial_{i}u^j+ u^k \partial_{j}u^i \partial_{k}\partial_{i}u^j)\Big)\,dx\\
&=\frac{d}{dt} \int_{\mathbb{T}^2}  (P\!-\!\langle P\rangle)  \nabla u: (\nabla u)^{\top} \,dx+\sum_{i,j=1,2}\int_{\mathbb{T}^2} \Big(\partial_{j}u^i \partial_{i} P \:\dot{u}^j  -(P\!-\!\langle P\rangle)(\partial_{j}\dot{u}^i \partial_{i}u^j\!+\!\partial_{j} u^i \partial_{i} \dot{u}^j ) \Big)\,dx\\
&\hspace{4cm}+2\sum_{i,j,k=1,2} \int_{\mathbb{T}^2}  (P-\langle P\rangle) \partial_{i} u^k \partial_{k} u^j \partial_{j}u^i\,dx.
\end{aligned}
\end{equation}
Taking advantage of  the duality between  ${\rm{BMO}} $ and $\mathcal{H}^1 $,  of Lemma \ref{lemmaBMO} as well as of \eqref{nablaunablaPL21} and of the fact that $\div u=0,$ we discover that
\begin{align*} 
\Bigg|\sum_{i,j=1,2}\int_{\mathbb{T}^2}
\!  \Big(\partial_{j}u^i \partial_{i} P \:\dot{u}^j 
  & -(P-\langle P\rangle)(\partial_{j}\dot{u}^i \partial_{i}u^j\!+\!\partial_{j} u^i \partial_{i} \dot{u}^j ) \Big)\,dx
  \Bigg|
  \\
& \leq \!C\sum_{j=1,2}\biggl(\|\dot{u}^j-\langle\dot u^j\rangle\|_{{\rm{BMO}}}\|\partial_{j}u\cdot\nabla P\|_{\mathcal{H}^1}
\!+\!\|P-\langle P\rangle\|_{{\rm{BMO}}}\|\partial_{j} u\cdot \nabla  \dot{u}^j\|_{\mathcal{H}^1}\biggr)\\
&\leq C\|\nabla\dot{u}\|_{L^2}\|\nabla P\|_{L^2}\|\nabla u\|_{L^2}\\
& \leq \frac{1}{8} \|\nabla\dot{u}\|_{L^2}^2+C \|\nabla u\|_{L^2}^2 \|\nabla P\|_{L^2}^2\\
& \leq \frac{1}{8} \|\nabla\dot{u}\|_{L^2}^2+C \|\nabla u\|_{L^2}^2 \|\dot{u}\|_{L^2}^2+C\|\nabla u\|_{L^2}^2 \|n_{f}\|_{L^{\infty}}\int_{\mathbb{T}^2\times\mathbb{R}^2}|u-v|^2f\,dxdv.
\end{align*}
Now, note that any $2\times 2$ matrix $A$ has the property
$$
{\rm Tr}\, A^3=({\rm Tr}\,A)^3-3 {\rm Tr}\, A \, {\rm Det}\, A.
$$
Hence, since ${\rm Tr}\,\nabla u=\div u=0$, one obtains
\begin{equation*}\sum_{i,j,k=1,2}\int_{\mathbb{T}^2} (P-\langle P\rangle) \partial_{i} u^k \partial_{k} u^j \partial_{j}u^i\,dx=\int_{\mathbb{T}^2} (P-\langle P\rangle) {\rm Tr}\, (\nabla u)^3\,dx=0.
\end{equation*}
Plugging the above estimates of $I_i$ ($i=1,...,5$) into \eqref{311} completes the proof of  \eqref{3111}.
\end{proof}



\subsection{Time-dependent regularity estimates}\label{subsectionlarge}

The following lemma provides us with time-dependent estimates for  $n_f,$  $j_f$ and $e_f$  as well as a control of $\|\nabla u(t)\|_{H^1}$ and $\|\nabla P\|_{L^2}$ for positive times, for (just) finite energy solutions of \eqref{homoVNS}.
It will be needed for proving Theorem \ref{theorem11}.

\begin{lemma}\label{c:ineq}
 Let $q>4$. For any given time $T>0$, it holds that 
  \begin{align}
& \sup_{t\in[0,T]}(\|n_f(t)\|_{L^{\infty}}+\|f(t)\|_{L^1_v(L^{\infty}_x)})\leq N_T,\label{nfT}\\
&\sup_{t\in[0,T]}(\|j_f(t)\|_{L^{\infty}}+\|vf(t)\|_{L^1_v(L^{\infty}_x)})\leq J_T,\label{JfT}\\
&\sup_{t\in[0,T]}(\|e_f(t)\|_{L^{\infty}}+\|v^2f(t)\|_{L^1_v(L^{\infty}_x)})\leq E_T,\label{efT}
 \end{align}
 where 
 \begin{equation}\label{eq:NJE}
 \left\{\begin{aligned}
     N_T&=\||v-u_\infty|^3f_0\|_{L^\infty_{x,v}}+e^{C T} e^{C\mathcal{H}_0}(1+\|f_0\|_{L^{\infty}_{x,v}}^3)\|f_0\|_{L^{\infty}_{x,v}},\\
     J_T&=\| |v-u_\infty|^4 f_0\|_{L^{\infty}_{x,v}}+|u_\infty| N_T+e^{C T} e^{C\mathcal{H}_0}(1+N_T)\|f_0\|_{L^{\infty}_{x,v}},\\
     E_T&=\| |v-u_\infty|^q f_0\|_{L^{\infty}_{x,v}}+|u_\infty|^2 N_T+e^{C T} e^{C\mathcal{H}_0}(1+N_T^{q/4})\|f_0\|_{L^{\infty}_{x,v}}.
 \end{aligned}\right.\end{equation}
 Furthermore, we have
 \begin{align}
  &\sup_{t\in[0,T]} \||u(t)-v|^2f(t)\|_{L^1_{x,v}}\leq C(1+N_T)\mathcal{H}_0, \label{Coro:1}\\
 &\sup_{t\in[0,T]}t \|\nabla u(t)\|_{L^2}^2+C^{-1}\int_0^Tt\|(\dot u,\nabla^2u,\nabla P)\|_{L^2}^2 \, dt \leq \bigl(1+C N_T\bigr) \mathcal{H}_0\: e^{C\mathcal{H}_0},\label{Coro:0}\\
&\sup_{t\in[0,T]}t^2\Big(\|\dot{u}(t)\|_{L^2}^2+\|\nabla^2 u(t)\|_{L^2}^2+\|\nabla P(t)\|_{L^2}^2\Big)+\int_0^Tt^2\|(\nabla \dot u,\sqrt{n_f}\dot{u})\|_{L^2}^2 \, dt \leq C_{T}.\label{Coro:2}
\end{align}
Here $C\geq 1$ depends only on $\T^2,$ and $C_{T}>0$ depends on $T$ and on the initial data.
\end{lemma}
\begin{proof}
Applying \eqref{Minfty} with $\bar u=u_\infty$ yields for all $t\in[0,T]$: 
\begin{equation}\label{eq:nf0}
\|n_f(t)\|_{L^{\infty}}\leq \|f(t)\|_{L^1_v(L^\infty_x)}\lesssim 
e^{-t}\||v-u_\infty|^3f_0\|_{L^\infty_{x,v}}+e^{2t}\bigl(1+\|u-u_\infty\|_{L^1_t(L^\infty)}^3\bigr)\|f_0\|_{L^\infty_{x,v}}.\end{equation}
By the Gagliardo-Nirenberg inequality, \eqref{eq:modulated} and \eqref{uinftybis}, we have
\begin{align}
\|u-u_\infty\|_{L^1_t(L^\infty)}
&\leq \|u-\langle u\rangle\|_{L^1_t(L^\infty_x)}+\|\langle u\rangle-u_\infty\|_{L^1_t}\nonumber\\
&\lesssim \int_0^t \|u-\langle u\rangle\|_{L^2}^{1/2} \bigl(\tau^{1/2}\|\nabla^2u\|_{L^2}\bigr)^{1/2}\tau^{-1/4}\,d\tau
+t^{1/2}\|\langle u\rangle-u_\infty\|_{L^2_t}\nonumber\\
&\lesssim t^{1/2}\mathcal{H}_0^{1/4} \biggl(\int_0^t\tau\|\nabla^2u\|_{L^2}^2\,d\tau\biggr)^{1/4}+t^{1/2}\mathcal{H}_0^{1/2}.\label{nabla2uqq}
\end{align}

Bounding the term  $\nabla^2u$ in  \eqref{nabla2uqq} stems from  Inequality \eqref{dpes021110000}: 
multiplying it by $t$ yields
$$
 \frac{d}{dt} \bigl(t\|\nabla u\|_{L^2}^2)+C^{-1}t\|(\dot{u},\nabla^2u,\nabla P)\|_{L^2}^2
\leq \|\nabla u\|_{L^2}^2+ Ct\|\nabla u\|_{L^2}^4+Ct\|n_{f}\|_{L^{\infty}} \int_{\mathbb{T}^2\times\mathbb{R}^2} |u-v|^2f\, dxdv.
 $$
Consequently, owing to the modulated energy identity \eqref{intenergym} and Gr\"onwall's lemma, we have
\begin{align}
 t\|\nabla u(t)\|_{L^2}^2+C^{-1}\int_0^t\tau\|(\dot u,\nabla^2u,\nabla P)\|_{L^2}^2\, d\tau \leq \bigl(1+Ct\|n_f\|_{L^\infty((0,t)\times\T^2})\bigr)
 \mathcal{H}_0\: e^{C\mathcal{H}_0}.\label{Coro:10}
\end{align}
Hence, combining \eqref{nabla2uqq} and \eqref{Coro:10}, we end up with
\begin{align}
    &\quad\|u-u_\infty\|_{L^1_t(L^\infty)}\leq C t^{{1}/{2}}\mathcal{H}_0^{1/2}e^{C\mathcal{H}_0} (1+t\|n_f\|_{L^\infty((0,t)\times\T^2)})^{{1}/{4}},\label{ddfgg}
\end{align}
and  reverting to \eqref{eq:nf0} gives 
\begin{multline*}
\|n_f(t)\|_{L^\infty}\leq \|f(t)\|_{L^1_v(L^\infty_x)}\lesssim 
e^{-t}\||v-u_\infty|^3f_0\|_{L^\infty_{x,v}}\\+e^{2t}\|f_0\|_{L^\infty_{x,v}}\Bigl(1+t^{3/2}e^{C\mathcal{H}_0}\mathcal{H}_0^{3/2}
\bigl(1+t\|n_f\|_{L^\infty((0,t)\times\T^2)}\bigr)^{3/4}\Bigr)\cdotp
\end{multline*}
Omitting $e^{-t}$ on the first term, and then leveraging  Young's inequality for  the last term, and the fact that  $r^k\leq C_ke^r$ for any $r\geq0$ and $k\geq0$ yields \eqref{nfT}. Next, plugging \eqref{nfT}  into \eqref{Coro:10} gives \eqref{Coro:0}.
Furthermore, arguing as in \eqref{particleed111}, one obtains from \eqref{intenergym} and \eqref{nfT} that
\begin{align*}
 \int _{\mathbb{T}^2\times\mathbb{R}^2} |v-u|^2 f&\,dxdv\\
&\leq 2 \int _{\mathbb{T}^2\times\mathbb{R}^2} \left|v-\frac{\langle j_{f}\rangle}{\langle n_{f}\rangle}\right|^2 f\, dxdv
+2\|n_f\|_{L^1}\left|\frac{\langle j_{f}\rangle}{\langle n_{f}\rangle}-\langle u\rangle\right|^2  +2 \int_{\mathbb{T}^2} |u-\langle u\rangle|^2 n_{f}\,dx\\
&\leq C(1+N_T)\mathcal{H}_0,
\end{align*}
whence \eqref{Coro:1}. 
\medbreak


 Bounding $vf$ is similar: we start with the observation (due to \eqref{Minfty}) that 
$$\begin{aligned}
\|vf(t)\|_{L^1_v(L^\infty_x)}&\!\leq\! |u_\infty|\|f(t)\|_{L^1_v(L^\infty_x)}+\|(v-u_\infty)f(t)\|_{L^1_v(L^\infty_x)}\\
&\!\leq\!|u_\infty|\|f(t)\|_{L^1_v(L^\infty_x)}
\!+\!C\bigl(e^{-2t}\||v-u_\infty|^4f_0\|_{L^\infty_{x,v}}\!+\!e^{2t}\bigl(1\!+\!\|u-u_\infty\|_{L^1_t(L^\infty)}^4\bigr)\|f_0\|_{L^\infty_{x,v}}\bigr)\cdotp
\end{aligned}$$
Hence, with the bounds \eqref{nfT} and  \eqref{ddfgg} at hand, we get \eqref{JfT}. 
 A similar argument may be applied for $|v|^2f$. 
 \smallbreak
Let us finally  justify \eqref{Coro:2}. Multiplying \eqref{3111} by $t^2$, we have
\begin{multline*}
\frac{d}{dt}\bigg(t^2\int_{\mathbb{T}^2} \Big(|\dot{u}|^2-(P-\langle P\rangle)\nabla u: (\nabla u)^{\top} \Big)\,dx\bigg)+ t^2\|\nabla \dot{u}\|_{L^2}^2+ t^2 \|\sqrt{n_f}\dot{u}\|_{L^2}^2\\
\leq 2t\int_{\mathbb{T}^2} \Big(|\dot{u}|^2-(P-\langle P\rangle)\nabla u: (\nabla u)^{\top} \Big)\,dx+ Ct^2\|\nabla u\|_{L^2}^2 \|\dot{u}\|_{L^2}^2\\+Ct^2\Big(1+\|e_f\|_{L^{\infty}}+\|n_{f}\|_{L^{\infty}}\bigl(1+ \|u\|_{L^{\infty}}^2\bigr)
+\bigl(\|n_{f}\|_{L^{\infty}}+\|n_{f}\|_{L^{\infty}}^2\bigr)\cD\Big) \||u-v|^2f\|_{L^1_{x,v}}.
 \end{multline*}
This leads to
\begin{multline}\label{337}
t^2\|\dot{u}\|_{L^2}^2+\int_0^t \tau^2\Big(\|\nabla \dot{u}\|_{L^2}^2+ \|\sqrt{n_f}\dot{u}\|_{L^2}^2\Big)\,d\tau\\
\leq  t^2\int_{\mathbb{T}^2}(P(t)-\langle P(t)\rangle)\nabla u(t):(\nabla u)^{\top}(t) \,dx+2\int_0^t \tau \int_{\mathbb{T}^2}(P-\langle P\rangle)\nabla u: (\nabla u)^{\top} \,dx\\
+ C\int_0^t\|\nabla u\|_{L^2}^2 (\tau^2\|\dot{u}\|_{L^2}^2)\,d\tau+C\int_0^t \tau^2 \Big(1+E_\tau +N_\tau (1+\|u\|_{L^{\infty}}^2)+(N_\tau+N_\tau^2)\mathcal{D}\Big)\mathcal{D}\,d\tau.
  \end{multline}
  Arguing as for proving \eqref{ddtnablauhomo3}, then using \eqref{nablaunablaPL21}, we have on $[0,T]$:
\begin{align}
\int_{\mathbb{T}^2}(P-\langle P\rangle)\nabla u: (\nabla u)^{\top} \,dx\nonumber
&\leq C\|\nabla P\|_{L^2}\|\nabla u\|_{L^2}^2\\&\leq
 C\|\nabla u\|_{L^2}^2\Bigl(\|\dot u\|_{L^2}+\sqrt{N_T}\||u-v|^2f\|_{L^1_{x,v}}^{1/2}\Bigr)\cdotp\label{3.37}\end{align}
Hence, using also  Young's inequality, we find that 
\begin{align*}
t^2\int_{\mathbb{T}^2}(P\!-\!\langle P\rangle)\nabla u: (\nabla u)^{\top} dx &\leq \frac{1}{2}t^2 \|\dot u\|_{L^2}^2+C\big(t\|\nabla u\|_{L^2}^2\big)^2+C t^2 N_T\||u-v|^2f\|_{L^1_{x,v}},\\
\int_0^t \!\!\tau\! \int_{\mathbb{T}^2}(P\!-\!\langle P\rangle)\nabla u: (\nabla u)^{\top} dx
&\leq C\sup_{\tau\in[0,t]}\|\tau^{1/2} \nabla u(\tau)\|_{L^2} \bigg(\int_0^t  \|\nabla u\|_{L^2}^2\,d\tau\bigg)^{1/2}\bigg(\int_0^t  \tau\|\dot{u}\|_{L^2}^2\,d\tau\bigg)^{1/2} \\
+Ct^{{1}/{2}} \sqrt{N_T} &\sup_{\tau\in[0,t]}\|\tau^{1/2} \nabla u(\tau)\|_{L^2} \|\nabla u\|_{L^2((0,t)\times\mathbb{T}^2)}\| |u-v|^2f\|_{L^1(0,t;L^1_{x,v})}^{1/2}.
\end{align*}
Therefore, reverting to \eqref{337} yields for all $t\in[0,T]$:
\begin{multline*}
t^2\|\dot{u}\|_{L^2}^2+\int_0^t \tau^2\Big(\|\nabla \dot{u}\|_{L^2}^2+ \|\sqrt{n_f}\dot{u}\|_{L^2}^2\Big)\,d\tau\lesssim
(1+N_T)\sup_{\tau\in[0,t]} \tau \|\nabla u(\tau)\|_{L^2}^2\\
+t^2N_T\sup_{\tau\in[0,t]}\| |u(\tau)-v|^2f(\tau)\|_{L^1_{x,v}}+ tN_T\|\nabla u\|_{L^2((0,t)\times\mathbb{T}^2)}^2\| |u-v|^2f\|_{L^1(0,t;L^1_{x,v})}
\\ + \Bigl(\sup_{\tau\in[0,t]} \tau \|\nabla u(\tau)\|_{L^2}^2\Bigr)^2+\biggl(\int_0^t  \|\nabla u\|_{L^2}^2\,d\tau\biggr)\biggl(\int_0^t  \tau\|\dot{u}\|_{L^2}^2\,d\tau\biggr)
\\+ C\int_0^t\|\nabla u\|_{L^2}^2 (\tau^2\|\dot{u}\|_{L^2}^2)\,d\tau
+ \int_0^t \tau^2 \Big(1+E_\tau +N_\tau (1+\|u\|_{L^{\infty}}^2)+(N_\tau+N_\tau^2)\mathcal{D}\Big)\mathcal{D}\,d\tau.
\end{multline*}
The first five terms of the right-hand side can be bounded by means of \eqref{intenergym} and \eqref{Coro:0},  
and the last one, by using that in addition, 
\begin{equation}\label{eq:uinfty}\|u\|_{L^{\infty}}\leq \langle u\rangle+C\|u-\langle u\rangle\|_{L^2}^{1/2}\|\nabla^2u\|_{L^2}^{1/2}
\leq |u_\infty|+|\langle u\rangle-u_\infty|+C\mathcal{H}^{1/4}\|\nabla^2u\|_{L^2}^{1/2}.\end{equation}
 Collecting the above estimates as well as \eqref{nablaunablaPL21}, \eqref{nfT}, \eqref{efT} and \eqref{Coro:0}, then 
 using Gr\"onwall lemma, we arrive at \eqref{Coro:2}.
\end{proof}

\subsection{Exponential decay rate in the case of small distribution functions}\label{subsectionexp}

 Let us fix some small enough $\eta>0$ (always less than $1$
 in what follows). We claim that one can find some time $t_\eta\geq1$ such that 
 \begin{equation}\label{eq:smallH} \mathcal{H}(t_\eta)\leq\eta\andf \cD(t_\eta)\leq\eta.\end{equation}
Indeed,  Proposition \ref{propalge} guarantees that the first part of \eqref{eq:smallH} is  satisfied for 
 \begin{equation}\label{tvar:in}
t_{\eta}\simeq \bigg(1+\frac{1+\mathcal{H}_0+\|f_{0} \log f_0\|_{L^1_{x,v}}}{\cM_0} \bigg) \biggl(\frac{\mathcal{H}_0}{\eta}\biggr)^{c_1 \cM_0}\,.
\end{equation}
Next, the modulated energy balance ensures that
$$t_\eta \:\min_{t\in[t_\eta,2t_\eta]} \cD(t)\leq \int_{t_\eta}^{2t_\eta}\cD(t)\,dt\leq \mathcal{H}(t_\eta)\leq\eta.$$
Since  $t_\eta\geq1$ if $\eta$ is small enough, one can conclude that there exists some $t\in[t_\eta,2 t_\eta]$ such that $\cD(t) \leq\eta.$ Still denoting (slightly abusively) this new time by $t_\eta,$ the full property \eqref{eq:smallH} is satisfied. 
 \medbreak
 The main aim of this subsection is to establish  various exponential decay rates
 \emph{from time $t_\eta$} for   smooth enough  finite energy   solutions.  
 For the time being, assume that the Lipschitz bound \eqref{eq:Lip} is satisfied for some $T>t_\eta$ and that
\begin{equation}\label{eq:N}
\sup_{t\in [t_\eta,T]}\|n_{f}(t)\|_{L^{\infty}}\leq N,\quad \sup_{t\in{[t_\eta,T]}}\|j_{f}(t)\|_{L^{\infty}}\leq J\andf \sup_{t\in [t_\eta,T]}\|e_{f}(t)\|_{L^{\infty}}\leq E,
\end{equation}
where the constants $N$, $J$ and $E$ will be chosen later.
\smallbreak
Based on Lemma \ref{c:ineq}, \eqref{eq:Lip} and \eqref{eq:N}, we shall focus on the proof of uniform estimates involving the modulated energy
$\mathcal{H}$, the dissipation rate $\cD$, and $\|\dot u\|_{L^2}$ on the time interval $[t_\eta,T]$.   
We shall eventually take advantage of these estimates and of some bootstrap argument 
to conclude that \eqref{eq:Lip} and \eqref{eq:N} are satisfied with $T=\infty,$ provided
$\|f_0\|_{L^\infty_{x,v}}$ is sufficiently small.

\vspace{1mm}

\begin{itemize}

\item {\textbf{Step 1: Exponential decay of the modulated energy}}.

\end{itemize}

\vspace{1mm}

The first step is to prove that the modulated energy tends exponentially fast to $0,$ with a rate that depends on $\cM_0$ and on $N$.
To do so, we observe that Inequality \eqref{particleed} and  Lemma \ref{lemmalogn}  imply that 
for some constant $C$ depending only on $\T^2,$ 
$$\mathcal H(t)\leq \int_{\mathbb{T}^2\times\mathbb{R}^2}|v-u|^2f\,dxdv+C\Bigl(1+\cM_0\log(1+N)\Bigr)\|\nabla u\|_{L^2}^2.$$
Hence, setting 
\begin{equation}\label{lambdaeta}\lambda=\frac{1}{2\max\biggl(1,C(1+\cM_0\log(1+N))\biggr)},\end{equation}
Inequality \eqref{intenergym}  gives
$$\frac d{dt}\mathcal H(t)+\lambda\mathcal H(t)+\frac12 \cD(t)\leq0\quad\hbox{on }\ [t_\eta,T],$$
whence
\begin{equation}\label{Eexphomo}
\sup_{t\in[t_\eta,T]} e^{\lambda (t-t_\eta)}\mathcal{H}(t)+\frac{1}{2}\int_{t_\eta}^{T}e^{\lambda (t-t_\eta)}\cD(t)\,dt\leq \mathcal{H}(t_\eta)\leq \eta.
\end{equation}

\vspace{1mm}

\begin{itemize}

\item {\textbf{Step 2: Exponential decay of $\cD(t)$}}.

\end{itemize}

\vspace{1mm}
Multiplying \eqref{dpes0211100001} by $e^{\lambda (t-t_\eta)}$ with $t\in[t_\eta,T]$ yields
\begin{multline*}
\frac{d}{dt}\Bigl(e^{\lambda (t-t_\eta)} \cD(t)\Bigr)+e^{\lambda (t-t_\eta)} \biggl(\|\dot{u}\|_{L^2}^2+\frac{1}{C}\|(\nabla^2u,\nabla P)\|_{L^2}^2)
+\||u-v|^2f\|_{L^1_{x,v}}\biggr)\\
\leq  (\lambda +CN)e^{\lambda  (t-t_\eta)}\cD +C\bigl(\|\nabla u\|_{L^2}^2+\|\nabla u\|_{L^{\infty}}\bigr) e^{\lambda (t-t_\eta)}\cD.
    \end{multline*}
Hence, using Gr\"onwall's lemma  
and Relation \eqref{intenergym}, we get
for all $t\in[t_\eta,T],$
\begin{multline*}
e^{\lambda (t-t_\eta)} \cD(t)+\int_{t_\eta}^te^{\lambda (\tau-t_{\eta})} \biggl(\|\dot{u}\|_{L^2}^2+\frac{1}{C}\|(\nabla^2u,\nabla P)\|_{L^2}^2+\||u-v|^2f\|_{L^1_{x,v}}\biggr)\,d\tau\\
\leq   Ce^{C\int_{t_\eta}^T(\|\nabla u\|_{L^2}^2+\|\nabla u\|_{L^{\infty}})\,d\tau}\Bigl(D(t_\eta)+(\lambda+N)\int_{t_\eta}^t e^{\lambda(\tau-t_\eta)}\mathcal{D}\, d\tau  \Bigr)\cdotp
    \end{multline*}
Taking \eqref{eq:Lip}, \eqref{eq:smallH}, \eqref{eq:N}, \eqref{Eexphomo} and Relation \eqref{intenergym} into account,
and remembering that $\lambda\leq1$ and $\eta\leq1,$  we conclude  that
\begin{multline}\label{eq:ExpD}
\sup_{t\in[t_\eta,T]} e^{\lambda (t-t_\eta)} \cD(t)+\int_{t_\eta}^T e^{\lambda (t-t_{\eta})} \bigl(\|(\dot{u},\nabla^2u,\nabla P)\|_{L^2}^2+\||u-v|^2f\|_{L^1_{x,v}}\bigr)\,dt
\leq   C (1\!+\!N) \eta.
    \end{multline}

\vspace{1mm}

\begin{itemize}

\item {\textbf{Step 3: Exponential decay of $\|\dot u\|_{L^2}$}}

\end{itemize}

\vspace{1mm}

Assuming from now on that  $N,$ $J$ and $E$ in \eqref{eq:N} have been chosen so that $1\leq N\leq \min(J,E),$
 Inequality \eqref{3111} reduces to 
\begin{multline*}
 \frac{d}{dt}\biggl(\|\dot{u}\|_{L^2}^2-\int_{\T^2}(P-\langle P\rangle)\nabla u: (\nabla u)^{\top}\,dx\biggr)
 + \|\nabla \dot{u}\|_{L^2}^2+  \|\sqrt{n_f}\dot{u}\|_{L^2}^2
\lesssim  \|\nabla u\|_{L^2}^2 \|\dot{u}\|_{L^2}^2\\
 +\Bigl(E+N\|u\|_{L^{\infty}}^2+N^2\cD\Bigr)\int_{\mathbb{T}^2\times\mathbb{R}^2} |u-v|^2 f\,dxdv.
\end{multline*}
Multiplying  by  $(t-t_\eta)e^{\lambda (t-t_\eta)/2}$ and integrating on $[t_\eta,t]$ for $t\in [t_\eta,T],$ we discover that
\begin{multline}\label{eq:dotu1}
(t-t_\eta)e^{\frac{\lambda (t-t_\eta) }2}\|\dot u(t)\|_{L^2}^2+
\int_{t_\eta}^t (\tau-t_\eta) e^{\frac{\lambda (\tau-t_\eta)}2}\bigl( \|\nabla \dot{u}\|_{L^2}^2+  \|\sqrt{n_f}\dot{u}\|_{L^2}^2\bigr)\,d\tau\\
\leq \int_{t_\eta}^t\Bigl(1+\frac\lambda2(\tau-t_\eta) \Bigr)e^{\frac{\lambda(\tau-t_\eta) }2}\|\dot u\|_{L^2}^2\,d\tau
+ C\int_{t_\eta}^t \cD(\tau) (\tau-t_\eta) e^{\frac{\lambda (\tau-t_\eta)  }2}\|\dot u\|_{L^2}^2\, d\tau\\
+ C\int_{t_\eta}^t(\tau-t_\eta) e^{\frac{\lambda (\tau-t_\eta) }2}\Bigl(E+N\|u\|_{L^{\infty}}^2+N^2\cD\Bigr)
\|f|u-v|^2\|_{L^1_{x,v}}\,d\tau\\
-\int_{t_\eta}^t\int_{\T^2}\Bigl(1+\frac\lambda2(\tau-t_\eta) \Bigr)e^{\frac{\lambda (\tau-t_\eta)  }2}(P-\langle P\rangle)\nabla u: (\nabla u)^{\top}\,dx d\tau\\
+(t-t_\eta)e^{\frac{\lambda (t-t_\eta)  }2}\int_{\T^2}(P(t)-\langle P(t)\rangle)\nabla u(t):\nabla u(t)\,dx.
\end{multline}
Arguing as for proving \eqref{3.37}, we easily get
\begin{align*}
&\int_{\T^2}(P-\langle P\rangle)\nabla u: (\nabla u)^{\top}\,dx\lesssim \|\nabla u\|_{L^2}^2\Bigl(\|\dot u\|_{L^2}+\sqrt N\||u-v|^2f\|_{L^1_{x,v}}^{1/2}\Bigr)\cdotp
\end{align*}
To continue the computations, one can bound $\|u\|_{L^\infty}$ by means of \eqref{eq:uinfty}, which gives, 
using also \eqref{intenergym}, \eqref{uinftybis}, \eqref{nablaunablaPL21} then Young's inequality 
to go from the second to the third line below
\begin{align*}
N\|u\|_{L^\infty}^2\||u-v|^2f\|_{L^1_{x,v}}&\lesssim N(|u_\infty|^2+\mathcal H) \||u-v|^2f\|_{L^1_{x,v}} 
+N\mathcal H^{1/2}\|\nabla^2 u\|_{L^2}\||u-v|^2f\|_{L^1_{x,v}}\\
\lesssim N(|u_\infty|^2 &+\mathcal H)\cD +N\mathcal H^{1/2}\|\dot u\|_{L^2}\||u-v|^2f\|_{L^1_{x,v}}
+N^{3/2}\mathcal H^{1/2}\||u-v|^2f\|_{L^1_{x,v}}^{3/2}\\
\lesssim N(|u_\infty|^2 &+\mathcal H)\cD +   \cD \|\dot u\|_{L^2}^2+(N^2\mathcal H+N\cD)\||u-v|^2f\|_{L^1_{x,v}}.\end{align*}
 Plugging  all the above inequalities in  \eqref{eq:dotu1} and using that $\lambda(\tau-t_\eta)  e^{\frac{\lambda (\tau-t_\eta)  }2}\leq C
  e^{\frac{2\lambda (\tau-t_\eta) }3},$  we discover that
\begin{multline*}
(t-t_\eta) e^{\frac{\lambda (t-t_\eta)}2}\|\dot u(t)\|_{L^2}^2+
\int_{t_\eta}^t (\tau-t_\eta) e^{\frac{\lambda (\tau-t_\eta)  }2}\bigl( \|\nabla \dot{u}\|_{L^2}^2+  \|\sqrt{n_f}\dot{u}\|_{L^2}^2\bigr)\,d\tau
\lesssim  \int_{t_\eta}^t e^{\frac{2\lambda(\tau-t_\eta)}3}\|\dot u\|_{L^2}^2\,d\tau\\+ \int_{t_\eta}^t \cD(\tau) (\tau-t_\eta)  e^{\frac{\lambda(\tau-t_\eta) }2}\|\dot u\|_{L^2}^2\, d\tau
+\int_{t_\eta}^t(\tau-t_\eta) e^{\frac{\lambda(\tau-t_\eta) }2}\Bigl(E+N|u_\infty|^2+N^2\mathcal H\Bigr)\cD(\tau)\,d\tau
\\
+\int_{t_\eta}^t (\tau-t_\eta)e^{\frac{\lambda(\tau-t_\eta)}2}N^2\cD(\tau)
\|f|u-v|^2\|_{L^1_{x,v}}\,d\tau\\
+\sup_{\tau\in [t_\eta,t]} \|\nabla u(\tau)\|_{L^2}^2 \bigg(\int_{t_\eta}^t  e^{\frac{2\lambda(\tau-t_\eta)}3}\,d\tau\bigg)^{1/2} \bigg(\int_{t_\eta}^t  e^{\frac{2\lambda(\tau-t_\eta)}3} \Big(\|\dot{u} \|_{L^2}^2+N \||u-v|^2f\|_{L^1_{x,v}}\Big)\,d\tau\bigg)^{1/2}\\
+ (t-t_\eta)e^{\frac{\lambda(t-t_\eta)}2}\|\nabla u(t)\|_{L^2}^2\Bigl(\|\dot u(t)\|_{L^2}+\sqrt N\||u(t)-v|^2f(t)\|_{L^1_{x,v}}^{1/2}\Bigr)\cdotp\end{multline*}
As we assumed that $N\geq1,$ the last term may be bounded from  \eqref{eq:ExpD} as follows: 
\begin{multline*}(t-t_\eta)e^{\frac{\lambda(t-t_\eta)}2}\|\nabla u(t)\|_{L^2}^2\Bigl(\|\dot u(t)\|_{L^2}+\sqrt N\||u(t)-v|^2f(t)\|_{L^1_{x,v}}^{1/2}\Bigr)
\lesssim N\eta\|\dot u(t)\|_{L^2}+N^2\eta^{3/2}.
\end{multline*}
Hence, using Gr\"onwall's lemma, as well as
 Inequalities \eqref{intenergym}, \eqref{eq:smallH}, \eqref{Eexphomo}, \eqref{eq:ExpD} and 
   $\lambda\leq 1,$ $E\geq1,$  and assuming that $\eta\leq N^{-1},$  we end up with
\begin{multline}\label{eq:Expdotu}
\sup_{t\in[t_\eta,T]}(t-t_\eta)e^{\frac{\lambda (t-t_\eta)}2}\|\dot u(t)\|_{L^2}^2+
\int_{t_\eta}^T (t-t_\eta)e^{\frac{\lambda (t-t_\eta)}2}\bigl( \|\nabla \dot{u}\|_{L^2}^2+  \|\sqrt{n_f}\dot{u}\|_{L^2}^2\bigr)\,d\tau
\\\leq C N\lambda^{-1}\biggl(E+N|u_\infty|^2+N^3\biggr) \eta.
\end{multline}
Remember that $u_\infty$ is given by \eqref{uinfty}, hence only depends on the initial data.

\vspace{1mm}

\begin{itemize}

\item {\textbf{Step 4: A Lipschitz estimate}}

\end{itemize}

\vspace{1mm}

In this step, we demonstrate that the desired Lipschitz bound follows from  Inequalities \eqref{eq:ExpD} and \eqref{eq:Expdotu}.
The starting point is the following embedding:
$$\|\nabla u\|_{L^\infty}\lesssim \|\nabla^2 u\|_{L^{3,1}},$$
where $L^{3,1}$ denotes the classical Lorentz space defined in e.g. \cite{Gra}. 
Now,  using again that the velocity equation may be seen as the Stokes equation \eqref{eq:stokes} and putting   
the above embedding together with Corollary \ref{corstokes}, we can write that for all $t\in[t_\eta,T],$ we have:
\begin{equation}\label{eq:Lip1}
\int_{t_\eta}^t\|\nabla u\|_{L^\infty}\,d\tau \lesssim \int_{t_\eta}^t\|\dot u\|_{L^{3,1}}\,d\tau +\int_{t_\eta}^t \biggl\|\int_{\R^2} f(u-v)\,dv\biggr\|_{L^{3,1}}\,d\tau.\end{equation}
By Gagliardo-Nirenberg's inequality, we have
\begin{equation}\label{eq:GNinfty}\|\dot u\|_{L^{3,1}}\lesssim 
 \|\dot u\|_{L^2}^{2/3}\|\nabla\dot u\|_{L^2}^{1/3}+ \|\dot u\|_{L^2}.\end{equation}
To handle the first term, we use H\"older inequality, \eqref{eq:ExpD} and \eqref{eq:Expdotu}, and get for all $t\in[t_\eta,T]$
\begin{align*}
\int_{t_\eta}^t &\|\dot u\|_{L^2}^{2/3}\|\nabla\dot u\|_{L^2}^{1/3}\,d\tau\\
&
=\int_{t_\eta}^t \Bigl(e^{\frac{\lambda (\tau-t_\eta)}4}\|\dot u\|_{L^2}\Bigr)^{2/3}
\Bigl((\tau-t_\eta)e^{\frac{\lambda (\tau-t_\eta)}4}\|\nabla\dot u\|_{L^2}\Bigr)^{1/3}(\tau-t_\eta)^{-1/3}e^{-\frac{\lambda (\tau-t_\eta)}4}\,d\tau\\
&\lesssim   \Bigl\|e^{\frac{\lambda (\tau-t_\eta)}4}\dot u\Bigr\|_{L^2(({t_\eta},t)\times\T^2)}^{2/3}
\Bigl\|(\tau-t_\eta)e^{\frac{\lambda (\tau-t_\eta)}4}\nabla\dot u\|_{L^2(({t_\eta},t)\times\T^2)}^{1/3}\biggl(\int_{t_\eta}^t (\tau-t_\eta)^{-2/3} e^{-\frac{\lambda (\tau-t_\eta)}2}\, d\tau\biggr)^{1/2}\\
&\lesssim \lambda^{-1/6}N^{1/6} \Bigl(1 +N\Bigr)^{1/3}\Bigl(1+E+N^{3}\Bigr)^{1/6}\eta^{1/2} \cdotp
\end{align*}
Here we used 
\begin{align*}
\int_{t_\eta}^t (\tau-t_\eta)^{-\frac23} e^{-\frac{\lambda (\tau-t_\eta)}2}\, d\tau=
 \int_{0}^{t-t_\eta} \tau^{-\frac23} e^{-\frac{\lambda\tau}2}\,d\tau\leq
 \lambda^{-1/3}\int_0^\infty\tau^{-\frac23}e^{-\frac\tau2}\,d\tau.
\end{align*}
To bound the term with just $\|\dot u\|_{L^2},$ one can use \eqref{eq:ExpD} as follows:
\begin{align*}
\int_{t_\eta}^t \|\dot u\|_{L^2}\,d\tau&\leq \|e^{\lambda (\tau-t_\eta)}\dot u\|_{L^2( (t_\eta,t)\times\T^2)}\bigg(\int_{t_\eta}^te^{-2\lambda (\tau-t_\eta)}d\tau \bigg)^{1/2}\lesssim \lambda^{-1/2}(1+N)^{1/2} \eta^{1/2}  
\cdotp
\end{align*}
To bound the term involving the $L^{3,1}$ norm of the Brinkman force, we argue as follows (where we use repeatedly that $N\geq1$ and $J\geq1$):
\begin{align*}
 \biggl\|\int_{\R^2} f(u-v)\,dv\biggr\|_{L^{3,1}}\!
 &\lesssim  \biggl\|\int_{\R^2} f(u-v)\,dv\biggr\|_{L^2}^{1/2} \biggl\|\int_{\R^2} f(u-v)\,dv\biggr\|_{L^6}^{1/2}\\
 &\lesssim \biggl\|\int_{\R^2} f(u-v)\,dv\biggr\|_{L^2}^{1/2} \Bigl(N^{1/2}\|u\|_{L^6}^{1/2}+\|j_f\|_{L^6}^{1/2}\Bigr)\\
  &\lesssim \biggl\|\int_{\R^2} f(u-v)\,dv\biggr\|_{L^2}^{1/2} \biggl(N^{1/2}\bigl(\|u-\langle u\rangle\|_{L^6}
  \!+\!|\langle u\rangle-u_\infty|\!+\!|u_\infty|\bigr)^{1/2}\!+\!J^{1/2}\biggr)\\
  &\lesssim  N^{1/4}\bigl\|f|u-v|^2\bigr\|_{L^1_{x,v}}^{1/4}  N^{1/2}\biggl(\|\nabla u\|_{L^2}^{1/2}
  +{\mathcal H}^{1/4}+|u_\infty|^{1/2}+J^{1/2}\biggr)\cdotp
  \end{align*}
Hence, using H\"older's inequality, \eqref{Eexphomo} and \eqref{eq:ExpD}, we get
\begin{align*}
\int_{t_\eta}^t&\biggl\|\int_{\R^2} \!f(u-v)\,dv\biggr\|_{L^{3,1}}\, d\tau\\
& \lesssim \!N^{3/4}\!\!\int_{t_\eta}^t\!\! e^{-\frac{\lambda (\tau-t_\eta)}4} \Bigl(e^{\lambda (\tau-t_\eta)} \bigl\|f|u-v|^2\bigr\|_{L^1_{x,v}}\Bigr)^{1/4}
  \Bigl(\|\nabla u\|_{L^2}^{1/2}  \!+\!{\mathcal H}^{1/4}\!+\!|u_\infty|^{1/2}\!+\!J^{1/2}\Bigr)\,d\tau\\
  &\lesssim \lambda^{-3/4}N^{3/4}\biggl(\int_{t_\eta}^te^{\lambda (\tau-t_\eta)} \bigl\|f|u-v|^2\bigr\|_{L^1_{x,v}}\,d\tau\biggr)^{1/4}
    \Bigl(\sup_{\tau\in[t_\eta,t]}\bigl(\cD\!+\!\mathcal{H}\bigr)^{1/4}(\tau)\!+\!|u_\infty|^{1/2}\!+\!J^{1/2}\Bigr)\\
     &\lesssim \lambda^{-3/4}N^{3/4}\Bigl(J^{1/2}+E^{1/4}+N^{1/4}|u_\infty|^{1/2}+N^{3/4}+|u_\infty|^{1/2}\Bigr)\eta^{1/4}\cdotp\end{align*}
  Putting the above inequalities together and remembering that $N,J,E\geq1,$ we end up with
  the following inequality for all $t\in[t_\eta,T]$:
  \begin{multline}\label{uLip}
\int_{t_\eta}^t\|\nabla u\|_{L^\infty}\,d\tau\lesssim \lambda^{-1/6}N^{1/2}\bigl(E+N^{3}\bigr)^{1/6} \eta^{1/2}+\lambda^{-1/2}N^{1/2} \eta^{1/2}\\
+\lambda^{-3/4}N^{3/4}\Bigl(J^{1/2}+E^{1/4}+N^{1/4}|u_\infty|^{1/2}+N^{3/4}\Bigr)\eta^{1/4}.
  \end{multline}

\vspace{1mm}

\vspace{1mm}

\begin{itemize}

\item {\textbf{Step 5: The bootstrap}}.

\end{itemize}

\vspace{1mm}


Let us fix some $\eta\in(0,1),$ define $t_\eta$ according to \eqref{tvar:in} and set 
  \begin{align*}
  N_0= \||v-u_\infty|^3f_0\|_{L^\infty_{x,v}}+1,\quad
 J_0=\| |v-u_\infty|^4 f_0\|_{L^{\infty}_{x,v}}+1
 \andf  E_0=\| |v-u_\infty|^q f_0\|_{L^{\infty}_{x,v}}+1.
  \end{align*}
   Lemma \ref{c:ineq} guarantees that there exists $\alpha_1=\alpha_1(t_{\eta},\mathcal{H}_0, \|(|v|^3+|v|^{q}f_0)\|_{L^{\infty}_{x,v}},u_{\infty})$ such that if  $f_0$ satisfies \eqref{smin1} then
\begin{equation}\label{eq:NJE0}   
    N_{t_\eta}\leq N_0,\quad J_{t_\eta}\leq J_0\andf E_{t_\eta}\leq E_0.
 \end{equation}
In this final step, we want to show that \eqref{eq:Lip} and \eqref{eq:N} are  satisfied for all  $T\in (t_\eta,\infty),$
with 
 $$  N=3N_{0},\quad J=J_{0}+3N_{0}(|u_{\infty}|+1)\andf E=E_{0}+5 N_{0}(|u_{\infty}|^2+1).$$
Let $T^*$ be defined by 
  $$ 
  T^*:=\sup\Big\{ t\in [t_\eta, T)~:~{\text{\eqref{eq:Lip} and \eqref{eq:N} hold true}}\Big\}\cdotp 
  $$
  If $N\eta\leq1,$ then the estimates of the previous steps are valid on $[t_\eta,T^*),$ with 
\begin{equation}\label{lambda0}
\lambda=\lambda_0:=\frac{1}{2\max\biggl(1,C(1+\cM_0\log(1+N_{0}))\biggr)}\cdotp\end{equation}
Since $N_0, J_0, E_0\geq1$, $\lambda_0\leq 1/2$ and  $N\eta\leq 1,$ we see  that \eqref{uLip} reduces to
 \begin{align*}
  \int_{t_\eta}^T\|\nabla u\|_{L^\infty}\,dt&\leq C\lambda_0^{-3/4}(\eta N)^{1/4}\Bigl(N^{1/2}J^{1/2}+N^{1/2}E^{1/4}+N|u_\infty|^{1/2}+N^{5/4}\Bigr)\\
  &\leq C \lambda_0^{-3/4}\eta^{1/4}N_0^{3/4}\Bigl(J_0^{1/2}+E_0^{1/4}+N_0^{1/4}|u_\infty|^{1/2}+N_0^{3/4}\Bigr)\cdotp
  \end{align*}
  Hence, there exists a (small) absolute constant $c\in(0,1)$ such that  
   \begin{equation}\label{eq:Lipestb}
  \int_{t_\eta}^T\|\nabla u\|_{L^\infty}\,dt \leq \frac1{20}\,,
  \quad\hbox{if } \    \eta\leq \eta_0:= c\lambda_0^{3}N_0^{-3}\Bigl(J_0^{2}+E_0+N_0|u_\infty|^2+N_0^3\Bigr)^{-1}\cdotp
  \end{equation}
It is clear that $\eta_0\leq1$ and $N\eta_0\leq 1$. Next, based on Lemma \ref{lemmaninfty}, \eqref{nfT} and \eqref{eq:Lip}, we immediately have 
\begin{align}
\sup_{t\in[t_\eta,T^*)} \|n_f(t)\|_{L^\infty}\leq 2 \|f(t_\eta)\|_{L^1_v(L^\infty_x)}\leq 2N_{0}<N.\label{NT*}
\end{align}
Using \eqref{jfinftylarge} with $\bar u=0,$ then \eqref{nfT} yields  for all $T\in(t_\eta,T^*),$    
$$
\begin{aligned} 
\|vf(t)\|_{L^1_v(L^\infty_x)}&\leq 2e^{-2(t-t_\eta)}\|vf(t_\eta)\|_{L^1_v(L^\infty_x)}
+2\|f(t_\eta)\|_{L^1_v(L^\infty_x)} \int_{t_\eta}^t e^{-(t-\tau)} \|u\|_{L^\infty}\,d\tau\\
 &\leq 2e^{-2(t-t_\eta)} J_0 + 2N_0\biggl( |u_\infty|+\int_{t_\eta}^t e^{-(t-\tau)} \|u-u_\infty\|_{L^\infty}\,d\tau\biggr)\cdotp\end{aligned}$$
 From \eqref{eq:uinfty}, we can write
 \begin{align*}
\int_{t_{\eta}}^{t}&e^{-(t-\tau)}\|u- u_{\infty} \|_{L^{\infty}}\,d\tau\\
&\lesssim\int_{t_{\eta}}^{t}e^{-(t-\tau)}\|u-\langle u\rangle\|_{L^{2}}^{\frac{1}{2}}\|\nabla^2 u\|_{L^2}^{\frac{1}{2}}\,d\tau
+\int_{t_{\eta}}^{t}e^{-(t-\tau)}\mathcal{H}^{1/2}(\tau)\,d\tau\\
&\lesssim \biggl(\sup_{[t_\eta,t]} e^{{\lambda_0(\tau-t_\eta)}}\|u-\langle u\rangle\|_{L^2}^{2}\biggr)^{\!1/4}
 \Bigl\|e^{\frac{\lambda_0(\tau-t_\eta)}4}\nabla^2 u\Bigr\|_{L^2((t_\eta,t)\times\T^2)}^{1/2}
\biggl(\int_{t_{\eta}}^{t}e^{-\frac43(t-\tau)}e^{-\frac{2\lambda_0}3(\tau-t_\eta)}\,d\tau\biggr)^{\!3/4}\\
&\quad+ \biggl(\sup_{\tau\in[t_\eta,t]} e^{\frac{\lambda_0}{2}(\tau-t_\eta)} \mathcal{H}^{1/2}(\tau)\biggr)\int_{t_{\eta}}^{t}e^{-(t-\tau)}e^{-\frac{\lambda_0}{2}(\tau-t_\eta)}\,d\tau.
\end{align*}
Hence, thanks to \eqref{Eexphomo} and \eqref{eq:ExpD}, and remembering that $\lambda_0\leq1/2$ and $N_0\geq1,$
\begin{equation}\label{eq:uinfty1}
\int_{t_{\eta}}^{t}e^{-(t-\tau)}\|u- u_{\infty} \|_{L^{\infty}}\,d\tau\leq CN_{0}^{1/4}
 e^{-\frac{\lambda_0}2(t-t_\eta)} \eta^{1/2}.\end{equation}
 Consequently, we have 
\begin{equation}\label{eq:jdecay}
 \|vf(t)\|_{L^1_v(L^\infty_x)}\leq   2e^{-2(t-t_\eta)} J_0 + 2N_0 \biggl(|u_\infty|
+CN_{0}^{1/4} e^{-\frac{\lambda_0}2(t-t_\eta)}\eta^{1/2}\biggr)\cdotp
\end{equation}
For bounding $|v|^2f,$ we start from \eqref{efinftylarge} with $\bar u=0$:
$$ \begin{aligned}
 \||v|^2f(t)\|_{L^1_v(L^\infty_x)}\leq 4 e^{-2(t-t_\eta)}\||v|^2 f(t_\eta)\|_{L^1_{v}(L^{\infty}_{x})}+8\|f (t_\eta)\|_{L^1_v(L^{\infty}_x)}  \bigg(\int_{ t_\eta}^{t} e^{- (t-\tau)} \|u\|_{L^{\infty}}\,d\tau\bigg)^2\\
\leq4 e^{-2(t-t_\eta)} E_0+ 4N_0\Biggl(|u_\infty|^2 + \biggl(\int_{ t_\eta}^{t} e^{- (t-\tau)} \|u-u_\infty\|_{L^{\infty}}\,d\tau\biggr)^2\Biggr)
\cdotp\end{aligned}$$
Then, we use \eqref{eq:uinfty1} to bound the last term, and  end up with 
\begin{equation}\label{eq:edecay}
\||v|^2 f(t)\|_{L^1_v(L^\infty)}\leq 4(e^{-2(t-t_\eta)}E_0+2N_0|u_\infty|^2) + CN_0\, e^{-\lambda_0(t-t_\eta)} N_0^{1/2}\eta.
\end{equation}
Since the definition of $\eta_0$ already ensures that $\eta_0 N_0^{1/2}\ll1,$ one can conclude that
up to a harmless change of $c$ in \eqref{eq:Lipestb}, 
 Inequalities  \eqref{eq:Lipestb}, \eqref{NT*}, \eqref{eq:jdecay} and \eqref{eq:edecay} are valid with $\eta=\eta_0,$
 which   ensures that \eqref{eq:Lip} and \eqref{eq:N} hold true with strict inequalities  on $[t_\eta,T^*).$
 Hence we must have   $T^*=\infty.$ In other words,  all the estimates of the previous steps hold true on the interval $[t_{\eta_0},\infty)$.
 
Note that an explicit value of $\alpha_1$ in \eqref{smin1} may be found from the definition of $\eta_0$ in \eqref{eq:Lipestb},  
 the definitions of $N_{t_{\eta_0}},$ $J_{t_{\eta_0}}$ and $E_{t_{\eta_0}}$  in \eqref{eq:NJE}, and
 the requirement that \eqref{eq:NJE0} has to be satisfied with $\eta=\eta_0$.


\subsection{Large-time exponential asymptotics}\label{subsectionexp1}

The computations that we performed so far readily ensure that
\begin{equation}\label{eq:decayff}
\mathcal{H}(t)+\|u(t)-u_{\infty}\|_{H^2}^2+\|\dot u(t)\|_{L^2}^2+\|\nabla P(t)\|_{L^2}^2\leq C e^{-\frac{{\lambda_0} t}2}
\quad\hbox{for all }\ t\geq t_{\eta},\end{equation}
with $\eta=\eta_0$, the constant $C$ depending only on suitable norms of the data, and $\lambda_0$ defined in \eqref{lambda0}. Since $t_\eta$ 
is defined in terms of the data, employing \eqref{energy} and Lemma \ref{c:ineq} with $T=t_\eta$ implies that the left-hand side of \eqref{eq:decayff} can be bounded on $[1,t_\eta],$ just in terms of suitable norms of the initial data. 
 Consequently, \eqref{eq:decayff} holds \emph{for all} $t\geq1$.
\medbreak
We claim that for all $t\geq 1$, we  also have
\begin{multline}\label{decayff}
W_1(f,n_{\infty}(x-u_{\infty}t)\otimes\delta_{v=u_{\infty}})+\| |v-u_{\infty}|^2 f(t)\|_{L^{1}_{x,v}}\\
+\|n_f(t)-n_{\infty}(\cdot-t u_{\infty})\|_{\dot{H}^{-1}}\leq C e^{-\frac{\lambda_0}2 t},
\end{multline}
where the profile $n_{\infty}\in \dot{H}^{-1}$ has been  defined in \eqref{eq:ninfty}.
\medbreak
For bounding the second term, we can use that owing to the definition of $u_\infty$ and to \eqref{uinftybis},
\begin{align*}
\| |v-u_{\infty}|^2 f\|_{L^{1}_{x,v}}&\leq 3\int_{\mathbb{T}^2\times\mathbb{R}^2}  \biggl|v-\frac{\langle j_f\rangle}{\langle n_f\rangle}\biggr|^2 f\,dxdv
+3\int_{\mathbb{T}^2\times\mathbb{R}^2} \biggl|\frac{\langle j_f\rangle}{\langle n_f\rangle}-\langle u\rangle\biggr|^2 f\,dxdv\\
&\hspace{6cm}+3\int_{\mathbb{T}^2\times\mathbb{R}^2} \bigl|\langle u\rangle -u_\infty\bigr|^2 f\,dxdv\\
 &\leq 6\mathcal{H}+3\cM_0\biggl|\frac{\langle j_f\rangle}{\langle n_f\rangle}-\langle u\rangle \biggr|^2
 +3\cM_0 \frac{\langle n_{f_0}\rangle^2}{(1+\langle n_{f_0}\rangle)^2} 
 \biggl|\langle u\rangle-\frac{\langle j_f\rangle}{\langle n_f\rangle}\biggr|^2\leq C\mathcal{H}.\end{align*}
Next, integrating  $\eqref{homoVNS}_1$ with respect to $(v,\tau)\in\mathbb{R}^2\times [0,t]$ yields
\begin{equation}
\widetilde{n}_f(t,x)=n_{f_0}(x)-{\rm{div}}_x \int_0^{t}\int_{\mathbb{R}^2}(v-u_{\infty})f(\tau,x+\tau u_\infty,v)\,dvd\tau,\label{nf0}
\end{equation}
where the shifted density $\widetilde n_f$ has been defined by
\begin{equation}\label{eq:wtnf} \widetilde{n}_f(t,x):=n_{f}(t,x+t u_{\infty}).\end{equation}  
The Cauchy-Schwarz inequality ensures that
$$
\biggl\|\int_{\mathbb{R}^2}(v-u_\infty)f\,dv\biggr\|_{L^2}\leq \|n_f\|_{L^\infty}^{1/2}\| |v-u_{\infty}|^2 f\|_{L^{1}_{x,v}}^{1/2}.$$
Hence,  due to \eqref{eq:decayff}, $\int_{\mathbb{R}^2}(v-u_{\infty})f\,dv$ converges exponentially fast to $0$ in $L^2,$ and we can  define
$$
j_{\infty}(x):=\int_0^{\infty}\!\!\!\int_{\mathbb{R}^2}(v-u_{\infty})f(\tau,x+\tau u_{\infty},v)\,dvd\tau\andf {n}_{\infty}(x):=n_{f_0}(x)-\div j_{\infty}(x)\in \dot{H}^{-1}.
$$
Therefore, \eqref{nf0} implies 
$$
\widetilde{n}_{f}(t,x)-{n}_{\infty}(x)=\div_{x} \int_{t}^{\infty}\int_{\mathbb{R}^2}(v-u_{\infty})f(\tau,x+\tau u_{\infty},v)\,dvd\tau,
$$
from which we infer that for all $t\geq t_{\eta}$,
\begin{equation}\label{nH-1}
\begin{aligned}
\|n_{f}(t)-n_{\infty}(\cdot-t u_{\infty})\|_{\dot H^{-1}}&=\|\widetilde{n}_{f}(t)-{n}_{\infty}\|_{\dot H^{-1}}\\
&\leq \int_{t}^{\infty}\left\|\int_{\mathbb{R}^2}(v-u_{\infty})f\,dv\right\|_{L^2}\,d\tau\\
&\leq \sup_{\tau\in[t,\infty)}\|n_{f}(\tau)\|_{L^{\infty}}^{\frac{1}{2}}\int_{t}^{\infty} \||v-u_{\infty}|^2{f}\|_{L^1_{x,v}}^{1/2}\,d\tau\\
&\leq C e^{-\frac{{\lambda_0}}{2}t}\cdotp
\end{aligned}
\end{equation}
Next, by virtue of Definition \ref{MKD}, we have
\begin{align*}
W_1(f,n_{f}\otimes\delta_{v=u_{\infty}})&=\sup_{\|\nabla_{x,v}\psi\|_{L^\infty}=1}\int_{\mathbb{T}^2\times\mathbb{R}^2} f(t,x,v)\left(\psi(x,v)-\psi(x,u_{\infty})\right) \,dxdv\\
&\leq\int_{\mathbb{T}^2\times\mathbb{R}^2} |v-u_{\infty}|f\,dxdv\\
&\leq\left(\int_{\mathbb{T}^2\times\mathbb{R}^2} f \,dxdv\right)^{\frac{1}{2}}\left(\int_{\mathbb{T}^2\times\mathbb{R}^2}|v-u_{\infty}|^2f \,dxdv\right)^{\frac{1}{2}}\leq C e^{-\frac{\lambda_0}2 t}\cdotp
\end{align*}
Together with \eqref{nH-1}, we get
\begin{align*}
W_1(f,n_{\infty}(x-u_{\infty}t)\otimes\delta_{v=u_{\infty}})
&\leq W_1(f,n_{f}\otimes\delta_{v=u_{\infty}})+W_1(n_{f}\otimes\delta_{v=u_{\infty}},n_{\infty}(x-t u_{\infty})\otimes\delta_{v=u_{\infty}})\\
&\leq W_1(f,n_{f}\otimes\delta_{v=u_{\infty}})+\|n_{f}-n_{\infty}(x-t u_{\infty})\|_{\dot H^{-1}}\\
&\leq C e^{-\frac{{\lambda_0}}{2}t}.
\end{align*}
Note that we used that the last term of the first line is controlled by the $\dot H^{-1}(\T^2)$ norm, 
since $\|\psi(\cdot,u_\infty)\|_{\dot H^1_x}\leq C\|\nabla_x\psi\|_{L^\infty}.$
This completes the proof of \eqref{decayff}.
\medbreak
Finally, we observe from the definition of $\widetilde n_f$ in \eqref{eq:wtnf} and the boundedness of $n_f,$
that $\widetilde n_f(t)$ is uniformly bounded on $L^\infty(\T^2).$ Hence, remembering that $\widetilde n_f(t)\to n_\infty$
as $t$ tends to $\infty,$ and using the standard compactness result for the
weak $*$ topology of $L^\infty$, one can conclude that $n_\infty$ is bounded.


\subsection{Construction of the solutions}\label{subsectionproofhomo}

For the reader's convenience, we here explain how to construct global solutions 
of  \eqref{homoVNS} with  the properties listed in Theorems \ref{theorem1} and \ref{theorem11}.  
\smallbreak
Let $(f_0,u_0)$ satisfy \eqref{a1}.  Arguing as in  \cite{Baranger2006}, 
 we regularize the initial data  (for all $k\in\N$)  as follows:
\begin{equation}\label{weakdata}
f_{0}^{k}(x,v):=J_{1}^{k}\ast J_{2}^{k}\ast(f_{0} \phi(|v|/k))(x,v),\quad u_{0}^{k}(x):=J_{1}^{k}\ast u_{0}(x),
\end{equation}
where $J_{1}^{k}$ and $J_{2}^{k}$ are  Friedrichs mollifiers with respect to the variables $x$ and $v$, respectively, and $\phi\in \cC_{c}^{\infty}(\mathbb{R};[0,1])$ is some  cut-off function  supported in $[-2,2]$ such that $\phi\equiv1$ on $[-1,1].$

For every $k\in\mathbb N$, one can show that there exists a maximal time $T_k$ such that System \eqref{homoVNS} supplemented with initial data $(f_0^k,u_0^{k})$ has a unique maximal strong  solution $(f^k,u^k,P^k)$ on the time interval $[0,T_k),$ which satisfies
for all $T\in(0,T_k)$:
\begin{equation*}
   0\leq  f^k\in \cC([0,T]; H^2_{x,v}),\quad u^k\in \cC([0,T]; H^2)\cap L^2(0,T; H^3)\andf \nabla  P^k\in L^2(0,T; H^1).
\end{equation*}
The proof can be done by a standard iteration process. Getting 
 higher order estimates of $f^k$ and $u^k$ on a short time interval does not present any particular difficulty.
 As we shall see below, the important point is that $f^k$ is compactly supported with respect to the $v$ variable so that integration in 
 the Brinkman term can be restricted to a bounded set.
 Being smooth,  $(f^k,u^k,P^k)$  satisfies all the estimates
that have been established so far, on $[0,T_k)$.  
 \smallbreak  
In order to show that the solution is global, 
let us  assume  by contradiction that  $T_k<\infty.$ In what follows, for any  $T\in(0,T_k),$ we denote by  $C_{k,T}$ a constant that  may depend on $k$ and $T$ but not on $t\in[0,T]$.
Now, since the solution is (reasonably) smooth, using the energy balance and 
following the lines of Lemma \ref{c:ineq}, \emph{but  without time weights}, one gets 
     \begin{equation*}
 \sup_{t\in[0,T]} \Big(\| u^k(t)\|_{H^2}+\|\dot u^k(t)\|_{L^2}+\|\nabla P^k(t)\|_{L^2}+\|(n_{f^k},j_{f^k},e_{f^k})(t)\|_{L^\infty}\Big)+
\|\dot u^k\|_{L^2(0,T;H^1)}\leq C_{k,T}.
    \end{equation*}
    Then, using \eqref{eq:Lip1} and \eqref{eq:GNinfty}, we arrive  at
     \begin{equation}\label{Lipk}
     \begin{aligned}
     \int_0^T \|\nabla u^k\|_{L^{\infty}}\,dt&
 \leq C_{k,T}.
     \end{aligned}
    \end{equation}
    Furthermore,  Lemma \ref{c:ineq} and suitable embedding ensure that
    $u^k\in L^\infty([0,T]\times{{\mathbb{T}^2}}).$ Hence, using Formula \eqref{repr1} and 
    Relation \eqref{repr2}, we discover that there exists $R_{k,T}>0$ depending only on $k,T$ and on the data such that
\begin{equation}
\text{{\rm Supp}}_{v}~f^k(t,x,v)\in \{ v\in \mathbb{R}^2~:~|v|\leq R_{k,T}\}.\label{compact}
\end{equation}
Next, differentiating $\eqref{homoVNS}_{1}$ with respect to $x_{i}$ for $i=1,2,$ and then taking the $L^2_{x,v}$ inner product of the resulting equation with $\partial_{x_i}f^k$, we get
\begin{align*}
\frac{1}{2}\frac{d}{dt}\|\partial_{x_i}f^k\|_{L^2_{x,v}}^2&=\int_{\mathbb{T}^2\times\mathbb{R}^2}\partial_{x_i}\left ((u^k-v)f^k \right) \nabla_{v} \partial_{x_i}f^k\,dxdv\\
&= \|\partial_{x_i}f^k\|_{L^2_{x,v}}^2-\int_{\mathbb{T}^2\times\mathbb{R}^2}\partial_{x_i} u^k\cdot\nabla_vf^k\partial_{x_i}f^k\,dxdv\\
&\leq \|\partial_{x_i}f^k\|_{L^2_{x,v}}^2+\|\nabla_x u^k\|_{L^{\infty}} \|\nabla_vf^k\|_{L^{2}_{x,v}} \|\partial_{x_{i}}f^k\|_{L^2_{x,v}}.
\end{align*}
After differentiating $\eqref{homoVNS}_{1}$ with respect to $v_{i}$, $i=1,2$, and taking the $L^2_{x,v}$ inner product
 with $\partial_{v_i}f^k$, one also deduces that
\begin{align*}
\frac{1}{2}\frac{d}{dt}\|\partial_{v_i}f^k\|_{L^2_{x,v}}^2&=\int_{\mathbb{T}^2\times\mathbb{R}^2}\left( -\partial_{x_i}f^k\partial_{v_i}f^k+\partial_{v_i}\left ((u^k-v)f^k \right) \nabla_{v} \partial_{v_i}f^k \right)\,dxdv\\
&\leq \|\partial_{x_i}f^k\|_{L^2_{x,v}}\|\partial_{v_i}f^k\|_{L^2_{x,v}}+2\|\partial_{v_i}f^k\|_{L^2_{x,v}}^2.
\end{align*}
Hence, it holds that for all $t\in[0,{{T}}],$ we have
$$\|(\nabla_{x}f^k,\nabla_{v}f^k)(t)\|_{L^2_{x,v}}\leq \|(\nabla_{x}f^k_0,\nabla_{v}f^k_0)\|_{L^2_{x,v}}+C\int_0^t
\bigl(1+\|\nabla_xu^k\|_{L^\infty}\bigr)\|(\nabla_{x}f^k,\nabla_{v}f^k)(\tau)\|_{L^2_{x,v}}\,d\tau,$$
whence, by Gr\"onwall's lemma and \eqref{Lipk}, 
\begin{equation*}\sup_{t\in[0,T]}\|(\nabla_{x}f^k,\nabla_{v}f^k)(t)\|_{L^2_{x,v}}\leq Ce^{CT+C\|\nabla u^k\|_{L^1(0,T;L^{\infty})}}\|(\nabla_{x}f^k_0,\nabla_{v}f^k_0)\|_{L^2_{x,v}}\leq C_{k,T}.\end{equation*}
Similarly, by differentiating the equations once more,  one can obtain that for any $0<T<T_k$,
\begin{equation*}
\sup_{t\in[0,T]}\Big(\|f^k(t)\|_{H^2_{x,v}}+\|u^k(t)\|_{H^2} \Big)\leq C_{k,T}. 
\end{equation*}
This uniform control enables us to solve  \eqref{homoVNS} with initial data  
 $(f^k,u^k)(t)$  from any time $t< T_{k},$  and  to continue  the solution $(f^k,u^k,P^k)$ beyond $T_k.$
  This contradicts the maximality of $T_k$. Therefore, we must have $T_k=\infty.$
  \smallbreak
  From this point, one can argue exactly as in \cite{Baranger2006} and conclude that 
  $(f^k,u^k,P^k)$ converges in the distributional meaning to some solution $(f,u,P)$ of System \eqref{homoVNS}.
  Furthermore, since all the bounds of the previous subsections are satisfied uniformly, they are also valid for 
  $(f,u,P).$ This completes the proof of both Theorems \ref{theorem1} and \ref{theorem11}.

\section{
The inhomogeneous case}\label{sectioninhomo}

This section aims to extend the decay results of the homogeneous incompressible Vlasov-Navier-Stokes system \eqref{homoVNS} to the inhomogeneous one \eqref{VNS}: we want to establish Theorems \ref{theorem1:in} and \ref{theorem11:in}. 
Since the overall approach is rather similar to what we did in the previous section, we just point out the main 
differences compared to the homogeneous case.
\medbreak 
Let us set 
\begin{equation*}
\rho_*:= \operatorname{ess}\inf\limits_{x\in\mathbb{T}^2}\rho_0(x)\geq 0\andf \rho^*:=
 \operatorname{ess}\sup\limits_{x\in\mathbb{T}^2}\rho_0(x)<\infty.\end{equation*}
The first  observation is that since $\rho$ satisfies a transport equation by a divergence-free vector field, we have, as long as the solution is defined,
 \begin{equation}\label{rhoupper}
\langle\rho(t)\rangle = \langle\rho_0\rangle \andf \rho_*\leq \rho(t,x) \leq \rho^*\ \hbox{ for a.e.}\ x\in\T^2.
\end{equation}
 Now,  Theorem \ref{theorem1:in} relies on the counterpart of  Proposition \ref{propalge} for solutions to \eqref{VNS}.
 We have to take care of the fact that   the modulated energy defined in \eqref{eq:modulated:in} also depends on $\rho$: instead of 
 \eqref{particleed111}, we write that
\begin{multline*}
\int _{\mathbb{T}^2\times\mathbb{R}^2} |v-u|^2 f\,dxdv
\geq \frac{1}{2} \int _{\mathbb{T}^2\times\mathbb{R}^2} \left|v-\frac{\langle j_{f}\rangle}{\langle n_{f}\rangle}\right|^2 f\, dxdv\\
+\frac{{\|n_f\|_{L^1}}}{2}\left|\frac{\langle j_{f}\rangle}{\langle n_{f}\rangle}-\frac{\langle \rho u\rangle}{\langle \rho\rangle}\right|^2  -3 \int_{\mathbb{T}^2} \left|u-\frac{\langle \rho u\rangle}{\langle \rho\rangle}\right|^2 n_{f}\,dx,
\end{multline*}
which obviously implies that 
\begin{equation}\label{eq:H}\mathbf{H}\leq  \int _{\mathbb{T}^2\times\mathbb{R}^2} |v-u|^2 f\,dxdv+\frac12 \int_{\mathbb{T}^2} \left|u-\frac{\langle \rho u\rangle}{\langle \rho\rangle}\right|^2\rho\,dx+3 \int_{\mathbb{T}^2} \left|u-\frac{\langle \rho u\rangle}{\langle \rho\rangle}\right|^2 n_{f}\,dx.
\end{equation}
 We observe that
\begin{align*}
\int_{\mathbb{T}^2} \rho \left|u-\frac{\langle \rho u\rangle}{\langle \rho\rangle}\right|^2\,dx
&\leq 2{{\int_{\mathbb{T}^2}}}\rho |u-\langle u\rangle|^2\,{{dx}}+\frac{2{{|\mathbb{T}^2|}}}{\langle\rho\rangle} 
|\langle\rho\rangle\langle u\rangle-\langle \rho u\rangle|^2.
\end{align*}
The first term of the right-hand side can be bounded from H\"older and Poincar\'e inequalities. 
To handle the second term, we combine Cauchy-Schwarz, Poincar\'e and H\"older inequalities and get 
\begin{equation}\label{eq:H2}
 |\langle\rho\rangle\langle u\rangle-\langle \rho u\rangle|^2\leq \|\rho\|_{L^2}^2\|u-\langle u\rangle\|_{L^2}^2
 \leq c_{\T^2}\|\rho\|_{L^2}^2\|\nabla u\|_{L^2}^2\leq c_{\T^2}  \|\rho\|_{L^1}\|\rho\|_{L^{\infty}}\|\nabla u\|_{L^2}^2.\end{equation}
In the end, we thus have 
 \begin{equation}
 \int_{\mathbb{T}^2} \rho \left|u-\frac{\langle \rho u\rangle}{\langle \rho\rangle}\right|^2\,dx
 \leq c_{\T^2}\rho^*\|\nabla u\|_{L^2}^2.\label{fluidenergy}
 \end{equation}
To bound the last term of \eqref{eq:H}, we use the fact that, owing to Young's inequality and \eqref{eq:H2},
 \begin{align}
  \int_{\mathbb{T}^2} \left|u-\frac{\langle \rho u\rangle}{\langle \rho\rangle}\right|^2 n_{f}\,dx&\leq
  2 \int_{\mathbb{T}^2} \left|u-\langle u\rangle\right|^2 n_{f}\,dx+2\| n_{f}\|_{L^1}  \left|{{\langle u\rangle}}-\frac{\langle \rho u\rangle}{\langle \rho\rangle}\right|^2\nonumber\\
  &\leq 2 \int_{\mathbb{T}^2} \left|u-\langle u\rangle\right|^2 n_{f}\,dx+c_{\T^2}\| n_{f}\|_{L^1}\rho^*\|\rho_0\|_{L^1}^{-1}\|\nabla u\|_{L^2}^2.  \label{aaabbbbb}
  \end{align}
  Bounding the first term by means of \eqref{u2f}, we eventually get 
  $$
  \int_{\mathbb{T}^2} \left|u-\frac{\langle \rho u\rangle}{\langle \rho\rangle}\right|^2 n_{f}\,dx\leq
  c_{\T^2}\|\nabla u\|_{L^2}^2\biggl(1+\|f_0\log f_0\|_{L^1} +\mathbf{M_0}(\rho^*\|\rho_0\|_{L^1}^{-1}+t)+\mathbf{H}_0\biggr)\cdotp
  $$
Reverting to \eqref{eq:H} and using \eqref{fluidenergy}, we conclude that 
  \begin{align}
  \mathbf{H}\leq c_{\T^2}\biggl(1+{\mathbf{H}_0+\mathbf{R_0}}+\|f_0\log f_0\|_{L^1} +\mathbf{M_0}t\biggr)\mathbf{D}
  \with \mathbf{R_0}:=\rho^*(1+\mathbf{M}_0\|\rho_0\|_{L^1}^{-1}). \label{Hinssss}
  \end{align}
Inserting \eqref{Hinssss} into the differential equality resulting from \eqref{energym:in} and arguing as in the proof of Lemma \ref{lemma51}, we conclude that \eqref{behavior1:in} is valid.

\medbreak 
Next, upgrading the algebraic {convergence rate}s to the exponential one in the case of small $f_0$ is based on the following adaptation 
of Lemma \ref{lemma32}. 
\begin{lemma}\label{lemmafunctionalin}
Let $(f,\rho,u,P)$ be a smooth solution to System \eqref{VNS} and assume that 
\begin{equation}\label{eq:novacuum}
\rho_*:={{\operatorname{ess}\inf\limits_{x\in\mathbb{T}^2}}}\rho_0(x)>0\andf  \rho^*:={{\operatorname{ess}\sup\limits_{x\in\mathbb{T}^2}}}\rho_0(x)<\infty.\end{equation}
There exists a constant $C>1$ depending only on $\T^2$ such that
the following properties hold:
\begin{itemize}
\item $H^1$-estimates{\rm:}
\begin{multline}\label{dpes021110000:in}
 \frac{d}{dt} \|\nabla u\|_{L^2}^2+\|\sqrt{\rho}\dot{u}\|_{L^2}^2+\frac{1}{C\rho^*}(\|\nabla^2u\|_{L^2}^2+\|\nabla P\|_{L^2}^2)\\
\leq C\|\nabla u\|_{L^2}^4+\frac C{\rho_{*}}\|n_{f}\|_{L^{\infty}} \int_{\mathbb{T}^2\times\mathbb{R}^2} |u-v|^2f\, dxdv
  \andf  \end{multline}
\begin{multline}\label{L1in}
\frac{d}{dt}\mathbf{D}+\|\sqrt{\rho}\dot{u}\|_{L^2}^2+\int_{\mathbb{T}^2\times\mathbb{R}^2} |u-v|^2f\, dxdv+\frac{1}{C\rho^*}(\|\nabla^2u\|_{L^2}^2+\|\nabla P\|_{L^2}^2)\\
\leq C\|\nabla u\|_{L^2}^4+C(\|n_{f}\|_{L^{\infty}}+\|\nabla u\|_{L^{\infty}})\int_{\mathbb{T}^2\times\mathbb{R}^2} |u-v|^2f\, dxdv.
\end{multline}
\item  {Material derivative estimate\rm:}
\begin{multline}
 \label{L2in}\frac{d}{dt}\int_{\mathbb{T}^2} (\rho|\dot{u}|^2-(P-\langle P\rangle)\nabla u: (\nabla u)^{\top})\,dx+ \|\nabla \dot{u}\|_{L^2}^2+  \|\sqrt{n_f}\dot{u}\|_{L^2}^2\\
\leq C\Bigl((\rho^*+(\rho_*)^{-1})\|\nabla u\|_{L^2}^2+\rho^*\|u\|_{L^\infty}^2\Bigr)\|\sqrt{\rho}\dot{u}\|_{L^2}^2 +C\bigg(1+\|e_f\|_{L^{\infty}}+\|n_{f}\|_{L^{\infty}}\bigl(1+ \|u\|_{L^{\infty}}^2\bigr)\\
+\bigl(\|n_{f}\|_{L^{\infty}}+\|n_{f}\|_{L^{\infty}}^2\bigr)\mathbf{D}\bigg) \int_{\mathbb{T}^2\times\mathbb{R}^2}|u-v|^2f\, dxdv.
\end{multline}
\end{itemize}
\end{lemma}
\begin{proof} We only indicate what has to be modified compared to the proof of Lemma \ref{lemma32}. In order to establish \eqref{dpes021110000:in}, we argue as for proving \eqref{dpes021110000}, first taking the $L^2$ scalar product
of the velocity equation with $\dot u$. Clearly, the former term $\|\dot u\|_{L^2}^2$ will become $\|\sqrt\rho \dot u\|_{L^2}^2$ and 
the Stokes equation \eqref{eq:stokes} now reads
 \begin{equation}\label{eq:stokesinh}
  -\Delta u+\nabla P=-\rho\dot{u}-\int_{\mathbb{R}^2}  (u-v)f\,dv.
  \qquad\div u=0,
 \end{equation}
 whence 
 \begin{equation}\label{eq:stokesinhb}
 \|\nabla^2 u\|_{L^2}^2+\|\nabla P\|_{L^2}^2
 \leq  C\rho^*\|\sqrt\rho\dot{u}\|_{L^2}^2+C\|n_f\|_{L^\infty}\int_{\mathbb{T}^2\times\mathbb{R}^2}
 |u-v|^2f\,dxdv.  \end{equation}
 This leads to \eqref{dpes021110000:in} (the appearance of $\rho_*^{-1}$ is due to the fact that $\dot u$ has to be eventually replaced by 
 $\sqrt\rho \dot u$ in \eqref{ddtnablauhomo4}).
 \medbreak
Proving  \eqref{L1in} is the same as for \eqref{dpes0211100001}.
For getting \eqref{L2in}, we observe that due to $\rho_t+\div(\rho u)=0$ and $\div u=0,$ we have instead of \eqref{rhorutt},
\begin{multline*}
\rho \partial_{t}\dot{u}^{j}- \Delta \dot{u}^{j} + n_{f} \dot{u}^{j}=-  
\partial_{i}( (\partial_{i}u\cdot\nabla) \nabla u^{j})- \div (\partial_{i}u \partial_{i}u^{j})\\
-\partial_{j}\partial_{t}P-(u\cdot\nabla )\partial_{j}P-\int_{\mathbb{R}^2}(u^j-v^j)(u-v)\cdot \nabla_{x}f\, dv+\int_{\mathbb{R}^2}(u^j-v^j)f\,dv.\end{multline*}
Consequently, when taking the $L^2$ scalar product with $\dot u^j$ and summing over $j=1,2$, a new term appears since
\begin{align*}
\int_{\T^2} \rho\dot u\cdot \partial_{t}\dot{u}\,dx&=\frac12\frac d{dt}\int_{\T^2}\rho|\dot u|^2\,dx
-\frac12\int_{\T^2}\rho_t|\dot u|^2\,dx\\
&=\frac12\frac d{dt}\int_{\T^2}\rho|\dot u|^2\,dx
+\frac12\int_{\T^2}\div(\rho u) |\dot u|^2\,dx\\&=\frac12\frac d{dt}\int_{\T^2}\rho|\dot u|^2\,dx
-\sum_{i,j}\int_{\T^2}\rho u^j\dot u^i\partial_j\dot u^i\,dx.
\end{align*}
The last term may be just bounded as follows:
$$-\sum_{i,j}\int_{\T^2}\rho u^j\dot u^i\partial_j\dot u^i\,dx\leq \frac14\|\nabla \dot u\|_{L^2}^2+\|\sqrt\rho\, u\|_{L^\infty}^2\|\sqrt\rho\,\dot u\|_{L^2}^2.
$$
Then, we have to bound all the terms $I_j$ in the right-hand side of \eqref{311} (their definition is unchanged).
For $I_2$ and $I_4,$ this is the same as in  the homogeneous case. 
For $I_1,$  we have to take care of the fact that the term $\|\nabla^2u\|_{L^2}$ in \eqref{eq:I1} has to be bounded 
by means of \eqref{eq:stokesinhb}, whence the term $\rho^*\|\nabla u\|_{L^2}^2\|\sqrt\rho u\|_{L^2}^2$ in \eqref{L2in}. 
The same occurs in $I_5$ when bounding $\|\nabla P\|_{L^2}.$
Finally, one can proceed as in the homogeneous case for handling $I_3$ except that, 
again, $\|\nabla^2 u\|_{L^2}$ has to be bounded according to \eqref{eq:stokesinhb} 
and that we have to use that
$$\sqrt{|\T^2|}\langle \dot u\rangle \leq \|\dot u\|_{L^2}\leq \frac1{\sqrt{\rho_*}}\|\sqrt\rho \dot u\|_{L^2}.$$
This completes the proof of  \eqref{L2in}.  
 \end{proof}



The next step in order to get exponential convergence estimates is to adapt  Lemma \ref{c:ineq} to the inhomogeneous setting.
By following its proof and observing that $\langle\rho\rangle^{-1}{\langle \rho u\rangle}-\bar u_\infty$ can be estimated by means of the modulated energy since 
\eqref{mass:in} holds and 
$$\frac{\langle \rho u\rangle}{\langle\rho\rangle}-\bar u_\infty =\frac{\langle n_f\rangle}{\langle n_f\rangle+\langle\rho\rangle}\Biggl(\frac{\langle \rho u\rangle}{\langle\rho\rangle}-\frac{\langle j_f\rangle}{\langle n_f\rangle}\Biggr),$$
we get the following result.
\begin{lemma}
Under Condition \eqref{eq:novacuum}, we have for any $q>4$ and given time $T>0$,  
  \begin{align}
& \sup_{t\in[0,T]}\|f(t)\|_{L^1_v(L^{\infty}_x)}\leq N_T^*,\quad \sup_{t\in[0,T]}\|vf(t)\|_{L^1_v(L^{\infty}_x)}\leq J_T^*, \quad \sup_{t\in[0,T]}\|v^2f(t)\|_{L^1_v(L^{\infty}_x)}\leq E_T^*\label{nfT:in},
 \end{align}
 where 
 \begin{align*}
     N^*_T&=\||v-\bar u_\infty|^3f_0\|_{L^\infty_{x,v}}+e^{C T} e^{C\mathbf{H}_0}(1+\|f_0\|_{L^{\infty}_{x,v}}^3)\|f_0\|_{L^{\infty}_{x,v}},\\
     J^*_T&=\| |v-\bar u_\infty|^4 f_0\|_{L^{\infty}_{x,v}}+|\bar u_\infty|N_T^*+e^{C T} e^{C\mathbf{H}_0}(1+N_T^*)\|f_0\|_{L^{\infty}_{x,v}},\\
     E^*_T&=\| |v-\bar u_\infty|^q f_0\|_{L^{\infty}_{x,v}}+|\bar u_\infty|^2 N_T^*+e^{C T} e^{C\mathbf{H}_0}(1+(N_T^*)^{q/4})\|f_0\|_{L^{\infty}_{x,v}}.
 \end{align*}
 Furthermore, there exists a constant $C\geq1$ depending on $\T^2$, $\|\rho_0\|_{L^1}$ and $\|(\rho_0,\rho_0^{-1})\|_{L^{\infty}}$ and a constant $C_T>0$ depending on $T$ and on the data such that
 \begin{align}
 &{\sup_{t\in[0,T]} \||u(t)-v|^2f(t)\|_{L^1_{x,v}}\leq C(1+N^*_T)\mathbf{H}_0, }\label{Coro:1:in0}\\
 &\sup_{t\in[0,T]}t \|\nabla u(t)\|_{L^2}^2+\int_0^T t \|(\sqrt\rho\,\dot u,\nabla^2u,\nabla P)\|_{L^2}^2 \, dt\leq C\bigl(1+ N^*_T\bigr) \mathbf{H}_0\: e^{C\mathbf{H}_0},\label{Coro:1:in}\\\
\label{Coro:2:in}
&\sup_{t\in[0,T]}t^2\Big(\|\sqrt\rho\,\dot{u}(t)\|_{L^2}^2+\|\nabla^2 u(t)\|_{L^2}^2+\|\nabla P(t)\|_{L^2}^2\Big)+\int_0^Tt^2\|(\nabla \dot u,\sqrt{n_f}\dot{u})\|_{L^2}^2 \, dt \leq C_{T}.
\end{align}
\end{lemma}
Let us now explain how to obtain an exponential decay rate for \eqref{VNS} in the case of a small distribution function and density bounded and bounded away from zero.  
Following the strategy of Section \ref{subsectionexp}, we establish suitable a priori bounds of $f$ and $u$. To this end, for some small constant $\eta^*\in (0,1)$ (to be chosen later), we fix a large time $t_{\eta^*}\geq1$  such that
\begin{equation}\label{4.8}
\mathbf{H}({t_{\eta^*}})\leq\eta^* \andf \mathbf{D}({t_{\eta^*}})\leq\eta^*.\end{equation}
Due to \eqref{behavior1:in}, there exists a small absolute constant $c_2$ such that  $t_{\eta^*}$  can be taken as
 \begin{equation}\label{tvar}
t_{\eta^*}\simeq \bigg(\frac{1+\mathbf{H}_0+\mathbf{R}_0+\|f_{0} \log f_0\|_{L^1_{x,v}}}{\mathbf{M}_0} \bigg) 
\biggl(\frac{\mathbf{H}_0}{\eta^*}\biggr)^{c_2 \mathbf{M}_0}\cdotp
\end{equation}
We now  suppose that $u$ satisfies the Lipschitz bound:
\begin{equation}\label{eq:Lip:in}
\int_{t_{\eta^*}}^T \|\nabla u\|_{L^{\infty}}\,dt\leq \frac{1}{10},
\end{equation}
and that $n_f$, $j_f$ and $e_f$ have the upper bounds (for some given 
$1\leq N^*\leq\min (J^*,E^*)$):
\begin{equation}\label{eq:N:in}
\sup_{t\in[t_{\eta^*},T]}\|n_{f}(t)\|_{L^{\infty}}\leq N^*,\quad \sup_{t\in[t_{\eta^*},T]}\|j_{f}(t)\|_{L^{\infty}}\leq J^*\andf \sup_{t\in[t_{\eta^*},T]}\|e_{f}(t)\|_{L^{\infty}}\leq E^*.
\end{equation}

We claim that under Condition \eqref{smin1:in}, Inequalities \eqref{eq:Lip:in} and \eqref{eq:N:in} are satisfied for all $T\in(t_{\eta^*}, \infty).$  
Since the proof is similar to that in Subsection \ref{subsectionexp}, we only give the key steps.
\smallbreak
 First, based on \eqref{mass:in}, \eqref{eq:N:in} and Lemma \ref{lemmalogn}, we have
\begin{align}
 \int_{\mathbb{T}^2}|u-\langle u\rangle|^2n_f\,dx\leq C(1+\mathbf{M}_0\log{(1+N^*)}) \|\nabla u\|_{L^2}^2.
 \end{align}
 This, together with \eqref{eq:H}, \eqref{fluidenergy} and \eqref{aaabbbbb}, leads to  
 \begin{equation}\label{lambdaeta*}\mathbf{D}\geq2 \lambda^* \mathbf{H}\with 
\lambda^*:=\frac{1}{C_{\mathbb{T}^2}(1+\mathbf{R}_0+\mathbf{M}_0\log(1+N^*))},
\end{equation}
for some constant $C_{\mathbb{T}^2}\geq1$ depending only on $\mathbb{T}^2$.
\smallbreak
Consequently, for all $t\geq t_{\eta^*}$ we have
$$\frac{d}{dt}\mathbf{H}(t)+\frac{1}{2}\mathbf{D}+\lambda^*\mathbf{H}\leq 0,$$
which leads to
\begin{equation}
\sup_{t\in[t_{\eta^*},T]} e^{\lambda^* (t-{t_{\eta^*}})}\mathbf{H}(t)+\frac{1}{2}\int_{t_{\eta^*}}^{T}e^{\lambda^*( t-{t_{\eta^*}})}\mathbf{D}(t)\,dt\leq \mathbf{H}({t_{\eta^*}})\leq \eta^*.\label{exp:inhomo}
\end{equation}
Next, based on \eqref{L1in} and \eqref{exp:inhomo}, there exists a constant $C>0$ that depends only on $\mathbb{T}^2$ and $\|(\rho_0,\rho_0^{-1})\|_{L^{\infty}}$ such that
\begin{multline}\label{Dlargein}
\sup_{t\in[t_{\eta^*},T]} e^{\lambda^* (t-t_{\eta^*})} \mathbf{D}(t)+\int_{t_{\eta^*}}^T e^{\lambda^* (t-t_{\eta^*})} \biggl(\|\sqrt{\rho}\dot{u}\|_{L^2}^2\\
+\|(\nabla^2u,\nabla P)\|_{L^2}^2+\||u-v|^2f\|_{L^1_{x,v}}\biggr)\,dt
\leq   C (1+N^*) \eta^*.
    \end{multline}
Together with \eqref{L2in}, this also leads to:
\begin{multline}\label{eq:Expdotu:in}
\sup_{t\in[t_{\eta^*},T]}(t-t_{\eta^*})e^{\frac{\lambda^* (t-t_{\eta^*})}2}\|\sqrt{\rho}\dot u(t)\|_{L^2}^2+
\int_{t_{\eta^*}}^T (t-t_{\eta^*})e^{\frac{\lambda^* (t-t_{\eta^*})}2}\bigl( \|\nabla \dot{u}\|_{L^2}^2+  \|\sqrt{n_f}\dot{u}\|_{L^2}^2\bigr)\,dt
\\\leq C N^*(\lambda^*)^{-1}\biggl(E^*+N^*|\bar{u}_\infty|^2+(N^*)^3\biggr) \eta^*.
\end{multline}
We omit the details of \eqref{Dlargein} and \eqref{eq:Expdotu:in} as they are very similar to those of the homogeneous 
case (see Steps 2--3 in Subsection \ref{subsectionexp}). Finally, we can use the time-dependent bounds \eqref{nfT:in} as well as the exponential convergence estimates  \eqref{exp:inhomo}--\eqref{eq:Expdotu:in} and argue  as in Steps 4--5 of Subsection \ref{subsectionexp}. Using Condition \eqref{smin1:in} and taking suitable constants $\eta^*$, $N^*$, $J^*$, $E^*$ and $\alpha_2$ that depend only on $\mathbb{T}^2$, $q$, $\mathbf H_0$, $\mathbf{M}_0$, $\bar{u}_{\infty}$, $\| (|v|^3+|v|^{q})f_0\|_{L^{\infty}_{x,v}}$, $\|\rho_0\|_{L^1}$ and $\|(\rho_0,\rho_0^{-1})\|_{L^{\infty}}$, we can justify that in fact, \eqref{eq:Lip:in} and \eqref{eq:N:in} are strict inequalities. 
Finally, a bootstrap argument shows that \eqref{eq:Lip:in}, \eqref{eq:N:in} and  \eqref{exp:inhomo}--\eqref{eq:Expdotu:in} are satisfied for all $t\in(t_{\eta^*},\infty)$.

In order to prove the large-time asymptotics of the solution (that is,  \eqref{2.24}), one can essentially  follow the lines of Subsection \ref{subsectionexp1}.  The only difference is that we have to specify additionally the behavior of $\rho$. 
Let us set 
$$\widetilde\rho(t,x):=\rho(t,x+\bar u_\infty t) \andf    \widetilde u(t,x):=u(t,x+\bar u_\infty t).$$
One can deduce from $\eqref{VNS}_2$ and $\eqref{VNS}_4$ that
$$\widetilde\rho_t+\div(\widetilde \rho(\widetilde u-\bar u_\infty))=0.$$
Hence, for all $t>0,$
$$\widetilde \rho(t,x)=\rho_0(x)-\div\int_0^t\widetilde \rho(\widetilde u-\bar u_\infty){(\tau,x)}\,d\tau.$$
Now, owing to the bound given by \eqref{2.24} for $u-\bar u_\infty,$ we may write for all $t>0,$
$$\|\widetilde \rho(\widetilde u-\bar u_\infty)(t)\|_{L^2}\leq \rho^*\|u(t)-\bar u_\infty\|_{L^2} \leq
Ce^{-\lambda_1t}.$$ This ensures that the integral converges in $L^2$ for $t$ going to infinity. 
Hence, one can define
$$
\bar\rho_\infty{(x)}:= \rho_0(x) -\div\int_0^\infty \rho (u-\bar u_\infty)(\tau,x+\bar{u}_{\infty}\tau)\,d\tau 
$$
as an element of $H^{-1}$ (that is also bounded by $\rho^*$ due to \eqref{rhoupper}) and get 
$$
\widetilde\rho(t)-\bar\rho_\infty=\div\int_t^\infty  \widetilde\rho (\widetilde u-\bar u_\infty)\,d\tau.
$$
Reverting to $\rho$ and using the exponential convergence of $u$ to $u_\infty$ gives $$
\|\rho(t)-\bar\rho_\infty(\cdot-\bar{u}_{\infty} t)\|_{H^{-1}}\leq Ce^{-\lambda_1 t},$$
as desired. 


Finally, both Theorems \ref{theorem1:in} and \ref{theorem11:in} can be shown by constructing a sequence of smooth solutions
after approximating the data, then using  compactness arguments based on the uniform estimates discussed before. Here we omit the details. To justify the existence of these smooth solutions, one can, for instance, take advantage of Theorem 1.1 in \cite{choiinhomo}, adapted to our choice of a drag force. \qed


\appendix

\section{Regularity estimates of the Vlasov equation}

In this section, we consider the following linear Vlasov equation in $\mathbb{T}^2\times\mathbb{R}^2$ 
\begin{equation}\label{V}
\left\{
\begin{aligned}
&\partial_{t}f+v\cdot\nabla_{x} f+\div_{v}( (u-v)f)=0,\\
&f|_{t=0}=f_0,
\end{aligned}
\right.
\end{equation}
where $u$ is a given time-dependent vector field in $L^1(0,T;W^{1,\infty})$.
\smallbreak
The following lemma, which is an easy adaptation 
of a result of \cite{goudonhy1}
reveals  that the $L^1 \log L^1$ `norm' of $n_f$ has at most linear time growth.

\begin{lemma}\label{lemmaflogf}
Let $f_0\in L^1_{x,v}$ be such that $f_0 \log{f_0}\in L^1_{x.v}$. Then, for all $\bar u\in\R^2,$ the solution $f$ to \eqref{V} satisfies
\begin{multline}
\int_{\mathbb{T}^2} n_{f}|\log n_{f}|\,dx\leq \| f_{0}\log f_{0} \|_{L^1_{x,v}}+(t+\log(2\pi))\|f_{0}\|_{L^1_{x,v}}\\+2e^{-1}|\mathbb{T}^2|
+\frac{1}{2}\int_{\mathbb{T}^2\times\mathbb{R}^2} |v-\bar u|^2f\,dxdv.\label{flogf}
\end{multline}
\end{lemma}

\begin{proof}

Multiplying \eqref{V} by $1+\log f$ and integrating by parts yields
\begin{equation}\nonumber
\begin{aligned}
\frac{d}{dt}\int_{\mathbb{T}^2\times\mathbb{R}^2} f\log f\, dxdv= \int_{\mathbb{T}^2\times\mathbb{R}^2} (u-v)\cdot \nabla_{v}f  \,dxdv=2  \int_{\mathbb{T}^2\times\mathbb{R}^2} f\, dxdv,
\end{aligned}
\end{equation}
which leads, owing to the total mass conservation \eqref{mass}, to 
\begin{equation}
\begin{aligned}
\int_{\mathbb{T}^2\times\mathbb{R}^2} f\log f\, dxdv\leq \int_{\mathbb{T}^2\times\mathbb{R}^2} f_{0}|\log f_{0}|\, dxdv+2   t\|f_{0}\|_{L^1_{x,v}}.\label{215}
\end{aligned}
\end{equation}
Consider  the convex function $h(s)=s\log{s},$ 
and $M(v)=\frac{1}{2\pi}e^{-\frac{|v-\bar u|^2}{2}}.$ 
As $\int_{\mathbb{R}^2}M(v)\,dv=1,$  Jensen's inequality guarantees that
\begin{equation}\nonumber
\begin{aligned}
h(n_{f})=h\left(\int_{\mathbb{R}^2}\frac{f}{M(v)}M(v)\,dv \right)&\leq \int_{\mathbb{R}^2}h\left(\frac{f}{M(v)}\right)M(v)\,dv\\
&=\int_{\mathbb{R}^2}\biggl(f\log{f}+\frac{|v-\bar u|^2}{2}f\biggr)\,dv+\log(2\pi) n_{f}.
\end{aligned}
\end{equation}
This, together with \eqref{mass} and \eqref{215}, gives
$$\int_{\mathbb{T}^2} n_{f} \log n_{f}\,dx\leq \int_{\mathbb{T}^2\times\mathbb{R}^2} f_{0}|\log f_{0}|\, dxdv+2 \|f_{0}\|_{L^1_{x,v}}  t+\log(2\pi)\|f_{0}\|_{L^1_{x,v}}+\frac{1}{2}\int_{\mathbb{T}^2\times\mathbb{R}^2} |v-\bar u|^2f\,dxdv.
$$
Since  $-s\log{s}\leq e^{-1}$ for any $s>0,$ we have 
\begin{equation}\nonumber
\begin{aligned}
n_{f}|\log n_{f}| &= n_{f}\log n_{f}  -2  n_{f}\log n_{f} \mathbb{I}_{0< n_{f}\leq 1}\leq n_{f}\log n_{f}+2e^{-1},
\end{aligned}
\end{equation}
which completes the proof of Lemma \ref{lemmaflogf}.
\end{proof}

\begin{lemma}\label{lemmaftime}
 Let $\bar{u}$ be any element of $\mathbb R^2.$
Assume that $(1+|v|^{q})f_0\in L^{\infty}_{x,v}$ with $q>4$.  Let $T>0$ be any given time. Then it holds that for all $t\in[0,T]$
and $r\in[0,q],$ 
 \begin{align}\label{finfty}
    &\|f(t)\|_{L^{\infty}_{x,v}}\leq e^{2 t}\|f_{0}\|_{L^{\infty}_{x,v}},\\
   \label{vqfinfty}
   &\||v-\bar{u}|^{r}f\|_{L^{\infty}_{x,v}}
    \leq \max(1,2^{r-1}) \Bigl(e^{(2-r)t} \||v-\bar u |^{r}f_{0}\|_{L^{\infty}_{x,v}}+ e^{2  t} \|u-\bar{u}\|_{L^1_{t}(L^{\infty})}^{r} \|f_{0}\|_{L^{\infty}_{x,v}}\Bigr)\cdotp
    \end{align}
 Furthermore, for all $t\in[0,T],$ $p\in[0,2]$ and $2<r\leq q-p,$  we have
\begin{align}
   &\Big\||v-\bar u|^p f\Big\|_{L^1_{v}(L^{\infty}_{x})}\leq C e^{(2-p-r)t} \||v-\bar u|^{r+p}f_{0}\|_{L^{\infty}_{x,v}}+C e^{2  t} (1+\|u-\bar{u} \|_{L^1_t(L^{\infty})}^{p+r})\|f_0\|_{L^{\infty}_{x,v}}.\label{Minfty}
    \end{align}
   \end{lemma}
\begin{proof}
  For any $(t,x,v)\in [0,T]\times \mathbb{T}^2\times\mathbb{R}^2$, let the characteristic curves $X(\tau;t,x,v)$ and $V(\tau;t,x,v)$ be defined by
\begin{equation}\label{char}
\left\{
\begin{aligned}
&\frac{d}{d\tau}X(\tau;t,x,v)=V(\tau;t,x,v),\\
&\frac{d}{d\tau}V(\tau;t,x,v)= u(\tau,X(\tau;t,x,v))-V(\tau;t,x,v),\\
&X(t;t,x,v)=x, \qquad V(t;t,x,v)=v.
\end{aligned}
\right.\qquad\tau\in[0,T]
\end{equation}
The solution of \eqref{V} is given by 
\begin{equation}\label{repr1}
f(t,x,v)=e^{2  t}f_0(X(0;t,x,v),V(0;t,x,v)),
\end{equation}
which implies \eqref{finfty}.
\medbreak

By \eqref{char}, we have the following formula for any $(x,v,t)\in \mathbb{T}^2\times\mathbb{R}^2\times[0,T]$:
\begin{equation}\label{repr2}
v=e^{-  t}V(0;t,x,v)+  \int_{0}^{t}e^{- (t-\tau)}  u(\tau,X(\tau;t,x,v))\,d\tau.
\end{equation}
 Hence, 
\begin{equation}\label{xv111}
\begin{split}
&v-\bar{u}=e^{-  t}\Big(V(0;t,x,v)-\bar{u} \Big)+  \int_{0}^{t}e^{- (t-\tau)}  \Big(u(\tau,X(\tau;t,x,v))-\bar{u}\Big)\,d\tau.
\end{split}
\end{equation}
We deduce from \eqref{repr1} and \eqref{xv111} that, for any $r\in[0,q]$, 
    \begin{multline*}
     f(t,x,v)|v-\bar{u}|^{r}
   \leq \max(1,2^{r-1})   \biggl(e^{- (r-2)t} f_{0}(X(0;t,x,v),V(0;t,x,v)) |V(0;t,x,v)-\bar{u}|^{r}\\
    +e^{2  t} f_{0}(X(0;t,x,v),V(0;t,x,v))\bigg( \int_{0}^{t}e^{- (t-\tau)} \|u(\tau)-\bar{u}\|_{L^{\infty}}\,d\tau \bigg)^{r}\biggr)
    \end{multline*}
which gives \eqref{vqfinfty}. Furthermore,  for all $r>2,$ we may write
     \begin{equation*}
    \begin{aligned}
   &\||v-\bar u|^p f\Big\|_{L^1_{v}(L^{\infty}_{x})} \leq  
   \biggl(\int_{\mathbb{R}^2} \frac{|v-\bar u|^p}{1+|v-\bar{u}|^{p+r}}\,dv\biggr) \Bigl(\|f\|_{L^{\infty}_{x,v}}+ \||v-\bar{u}|^{p+r}f\|_{L^{\infty}_{x,v}}\Bigr),
    \end{aligned}
    \end{equation*}
which implies \eqref{Minfty}. 
\end{proof}

The following result is an adaptation of  \cite[Lemma 4.5]{han1}.

\begin{lemma}\label{lemmaninfty}
Let $f$ be a solution to the Vlasov equation \eqref{V} on $[0,T]$ with $T>t_{*}\geq0$. If
\begin{equation}\label{prioriff}
\begin{aligned}
 \int_{t_{*}}^{T}\|\nabla u\|_{L^{\infty}}\,dt\leq \frac{1}{10},
\end{aligned}
\end{equation}
  then, for all $t\in[t_{*},T]$  and  element $\bar{u}$ of $\mathbb R^2,$  we have
\begin{align}\label{nfinftylarge}
\|n_f(t)\|_{L^\infty}\leq  \|f (t)\|_{L^1_v(L^{\infty}_x)}
&\leq 2 \|f (t_{*})\|_{L^1_v(L^{\infty}_x)},\\\label{jfinftylarge}
  \left\|\int_{\mathbb{R}^2}|v-\bar{u}| f(t)\,dv\right\|_{L^{\infty}}&\leq  2e^{-2(t-t_*)}\||v-\bar{u}| f(t_*)\|_{L^1_{v}(L^{\infty}_{x})}\\
 &\qquad+4\|f (t_{*})\|_{L^1_v(L^{\infty}_x)} \int_{ t_{*}}^{t} e^{- (t-\tau)} \|u(\tau)-\bar{u}\|_{L^{\infty}}\,d\tau,\nonumber   \\
 \label{efinftylarge}
 \left\|\int_{\mathbb{R}^2}|v-\bar{u}|^2 f(t)\,dv\right\|_{L^{\infty}}&\leq 4e^{-2(t-t_*)}\||v-\bar{u}|^2 f(t_*)\|_{L^1_{v}(L^{\infty}_{x})}\\
 &\qquad+8\|f (t_{*})\|_{L^1_v(L^{\infty}_x)}  \bigg(\int_{ t_{*}}^{t} e^{- (t-\tau)} \|u(\tau)-\bar{u}\|_{L^{\infty}}\,d\tau\bigg)^2\cdotp\nonumber
\end{align} 
\end{lemma}
\begin{proof}
Leveraging \eqref{char}, we see that the solution of \eqref{V} satisfies
\begin{equation}\label{repr11}
f(t,x,v)=e^{2  (t- t_{*})}f( t_{*},X( t_{*};t,x,v),V( t_{*};t,x,v))\quad\!\!\hbox{for all  } (t,x,v)\in [ t_{*},T]\times\mathbb{T}^2\times\mathbb{R}^2.
\end{equation}
Hence  it holds for any $(t,x)\in[ t_{*},T]\times\mathbb{T}^2 $ that
\begin{equation}\label{nfrep0}
n_f(t,x)=e^{2  (t- t_{*})}\int_{\mathbb{R}^2} f( t_{*},X( t_{*};t,x,v),V( t_{*};t,x,v))\, dv.
\end{equation}
In order to establish \eqref{nfinftylarge}, it suffices 
 to show that the map  $\Gamma_{t,x}:v\mapsto V( t_{*};t,x,v)$ is a $C^1$-diffeomorphism, and to get a suitable 
  control on its Jacobian
determinant. 

Now,  based on \eqref{char}, we have
\begin{equation*}
\frac{d}{d\tau} D_{x,v}Z(\tau;t,x,v)=D_{x,v}W(\tau, Z(\tau;t,x,v))\cdot  D_{x,v}Z(\tau;t,x,v),
\end{equation*}
where $Z:=(X,V)$ and $W:=(v, u-v)$. By Gr\"onwall's inequality, we deduce that 
\begin{equation*}
\| D_{x,v}Z(\tau;t,\cdot,\cdot)\|_{L^{\infty}_{x,v}}\leq \| D_{x,v}Z(t)\|_{L^{\infty}_{x,v}} e^{\int^{t}_{\tau} \|D_{x,v}W\|_{L^{\infty}}\,ds}. 
\end{equation*}
As $D_{x,v}Z(t)={\rm Id}$, this implies that for $ t_{*}\leq \tau\leq t\leq T$,
\begin{equation}\label{Xvinfty}
\|\nabla_{x,v}Z(\tau;t,\cdot,\cdot)\|_{L^{\infty}_{x,v}}\leq \exp{\left( (t-\tau)+ \int_{\tau}^{t}\|\nabla u(s)\|_{L^{\infty}}\,ds\right)}\cdotp
\end{equation}
In addition, integrating $\eqref{char}_{2}$ over $[ t_{*},t]$ yields
\begin{equation}\label{soluv}
e^{  t}v-e^{   t_{*}}V( t_{*};t,x,v)=  \int_{ t_{*}}^{t}e^{  \tau}u(\tau,X(\tau;t,x,v))\,d\tau,\quad\quad  t_{*}\leq t\leq T.
\end{equation}
Then, taking the derivative of \eqref{soluv} with respect to $v$, we have
\begin{equation}\nonumber
e^{  t}{\rm{Id}}-e^{   t_{*}} D_v\Gamma_{t,x}(v)= \int_{ t_{*}}^{t}e^{ \tau} \nabla u(\tau,X(\tau;t,x,v))\cdot D_vX(\tau;t,x,v)\,d\tau,
\end{equation}
which together with \eqref{Xvinfty} gives
\begin{equation*}
\begin{aligned}
\|e^{-  (t- t_{*})}D_v\Gamma_{t,x}(v)-{\rm{Id}}\|_{L^\infty_{x,v}}&\leq \int_{ t_{*}}^{t}e^{- (t-\tau)}\|\nabla u(\tau)\|_{L^\infty}\|D_vX(\tau)\|_{L^\infty_{x,v}}\,d\tau\\
&\leq \Biggl(\int_{ t_{*}}^t\|\nabla u(\tau)\|_{L^\infty}\,d\tau\Biggr)\exp{\left( \int_{ t_{*}}^t\|\nabla u(\tau)\|_{L^\infty}\,d\tau\right)\cdotp}
\end{aligned}\end{equation*}
Since \eqref{prioriff} holds and $e^{1/10}/10<1/9$, we  deduce from 
\cite[Lemma 9.4]{han1} that
\begin{equation}
|\det D_v\Gamma_{t,x}(v)|\geq \frac{1}{2}e^{2  (t- t_{*})}.\label{diffeomorphism}
\end{equation}
Now, performing the change of variable $v=\Gamma_{t,x}^{-1}(w)$ in \eqref{nfrep0} yields
\begin{equation}\label{nfrep}
n_f(t,x)=e^{2  (t- t_{*})}\int_{\mathbb{R}^2} f( t_{*},X( t_{*};t,x,\Gamma_{t,x}^{-1}(w)),w)\left|\det D_v\Gamma_{t,x}(\Gamma_{t,x}^{-1}(w))\right|^{-1}\,dw
\end{equation}
which, combined with \eqref{diffeomorphism}, yields \eqref{nfinftylarge}.
\medbreak
Similarly, we get from  \eqref{xv111} (after replacing $0$ by $t_*$),  and  \eqref{repr11} that
\begin{multline*}
\int_{\mathbb{R}^2} |v-\bar{u}|^2 f(t,x,v)\,dv\\
\leq 2 e^{-2 (t- t_{*})}\int_{\mathbb{R}^2} |V( t_{*};t,x,v)-\bar{u} |^2 f( t,x,v)\,dv+2  n_{f}(t,x) \bigg(\int_{ t_{*}}^{t} e^{- (t-\tau)} \|u-\bar{u}\|_{L^{\infty}}\,d\tau\bigg)^2. 
\end{multline*}
In view of \eqref{diffeomorphism}, we have
$$
\begin{aligned}
\int_{\mathbb{R}^2} |V( t_{*};t,x,v)-\bar{u} |^2&f( t,x,v)\,dv\\
&=  e^{2 (t- t_{*})}\int_{\mathbb{R}^2} | V( t_{*};t,x,v)-\bar{u} |^2f( t_{*},X( t_{*};t,x,v),V( t_{*};t,x,v)) \, dv\\
&\leq 2 \int_{\mathbb{R}^2} |w-\bar{u}|^2 f( t_{*}, X( t_{*};t,x,\Gamma_{t,x}^{-1}(w)), w)\,dw\\
&\leq 2 \||v-\bar{u}|^2 f(t_*)\|_{L^1_{v}(L^{\infty}_{x})},
\end{aligned}$$
which gives \eqref{efinftylarge}.
 Proving \eqref{jfinftylarge}  is similar.
\end{proof}


\section{Technical lemmas}

Let us first recall the following result that has been proved in \cite{CLMS}.
\begin{lemma}\label{lemmaBMO}
Let $\mathcal{H}^1 $ and ${\rm{BMO}}$ denote  the usual Hardy and Bounded Mean Oscillation spaces. Then, the following statements hold:
\begin{itemize}
\item {\rm(1)} For all vector fields $B$ and $E$ with coefficients in $L^2(\T^2)$
such that $\div E=0$ and $\nabla\times B=0,$  we have
$$\|E\cdot B\|_{\mathcal{H}^1}\leq C\|E\|_{L^2}\|B\|_{L^2}.$$
\item {\rm(2)}  For all $v\in H^1(\T^2)$, we have
\begin{equation}\nonumber
\begin{aligned}
&\|v\|_{{\rm{BMO}}}\leq C\|\nabla v\|_{L^2}.
\end{aligned}
\end{equation}
\end{itemize}
\end{lemma}

We need the classical maximal regularity properties of the Stokes system (see e.g. \cite{danchin13inhomo}).

\begin{lemma}\label{lemmastokes}
Let $g$ be a mean free function of $L^{p}(\mathbb{T}^d)$ with $1<p<\infty$. If $(u,P)$ is a solution to the Stokes equation
\begin{equation}\label{eq:stokes0}-\Delta u+\nabla P=g \quad\hbox{and}\quad {\rm{div}}\, u=0
\quad\hbox{in }\  \mathbb{T}^d\end{equation}
then, there exists some positive constant $C$ depending only on $p$ and ${\mathbb T}^d$ such that 
\begin{equation}\nonumber
\|\nabla^2 u\|_{L^p}+\|\nabla P\|_{L^p}\leq C\|g\|_{L^{p}}.
\end{equation}
\end{lemma}
Taking advantage of the real interpolation theory (see e.g. \cite{Gra}), one can extend the previous results to Lorentz spaces 
as follows:
\begin{coro}\label{corstokes}
Let $g$ be a mean free function of  $L^{p,r}(\mathbb{T}^d)$ with $1<p<\infty$ and  $1\leq r\leq\infty$. 
If $(u,P)$ satisfies \eqref{eq:stokes0}, then 
 there exists some positive constant $C$ depending only on $p$ and ${\mathbb T}^d$ such that 
\begin{equation*}\|\nabla^2 u\|_{L^{p,r}}+\|\nabla P\|_{L^{p,r}}\leq C\|g\|_{L^{p,r}}.\end{equation*}
\end{coro}

There are several equivalent definitions of the Wasserstein distances
(see e.g. Villani's book \cite{villani2}).
The one that is used in this paper is the following:
\begin{defn}\label{MKD} Let $X$ be a metric space.
For all pair $(\mu_1,\mu_2)$ of  Borel measures on $X$, we set:
\begin{equation}
W_1(\mu_1,\mu_2)=\sup\left\{\bigg|\int_X\psi(z)\,d\mu_1(z)-\int_X\psi(z)\,d\mu_2(z)\bigg|,\:\psi\in{\rm{Lip}}(X),~\|\nabla\psi\|_{L^\infty(X)}= 1\right\}\cdotp\nonumber
\end{equation}\end{defn}

Finally, the following  logarithmic  estimate was useful for investigating the inhomogeneous Vlasov-Navier-Stokes equations.
\begin{lemma}\label{lemmalogn}
Let $g\in H^1 $ and $0\leq h\in L^1\cap L^{\infty} $. Then, it holds that
\begin{equation}
\begin{aligned}
\int_{\mathbb{T}^2} |g-\langle g\rangle|^2h\,dx\leq  C\|\nabla g\|_{L^2}^2 \Big(1+\|h\|_{L^1}\log(1+\|h\|_{L^{\infty}}) \Big)\cdotp\label{logg}
\end{aligned}
\end{equation}
\end{lemma}
\begin{proof}Assume for simplicity that $\T^2$ is the unit torus. 
With no loss of generality, one can suppose that  $\langle g\rangle=0$. 
Then, we decompose $g$ into Fourier series as follows:
\begin{align}
&g=\sum_{1\leq |k|\leq n}\widehat{g}_k e^{2i\pi k\cdot x}+\sum_{|k|\geq n+1}\widehat{g}_k e^{2i\pi k\cdot x},\quad n\in{\mathbb N}.\label{decomg}
\end{align}
On the one hand, by the Cauchy-Schwarz inequality and the Fourier-Plancherel theorem, we have
\begin{equation}\nonumber
\begin{aligned}
\biggl|\sum_{1\leq  |k|\leq n}\widehat{g}_k e^{2i\pi k\cdot x}\biggr|^2&\leq \biggl(\sum_{1\leq |k|\leq n}|2\pi k\widehat{g}_k |\frac{|e^{2i\pi k\cdot x}|}{2\pi |k|}\biggr)^2\\
&\leq C\sum_{k\in{\mathbb Z}^2}|k|^2|\widehat{g}_k|^2\sum_{1\leq |k|\leq n}\frac{1}{|k|^2}\leq C{\log n} \| \nabla g\|_{L^2}^2,
\end{aligned}
\end{equation}
which implies
\begin{equation}\nonumber
\begin{aligned}
&\int_{\mathbb{T}^2} \bigg|\sum_{1\leq |k|\leq n}\widehat{g}_k e^{2i\pi k\cdot x}\bigg|^2h\,dx\leq C \log n  \|h\|_{L^1} \| \nabla g\|_{L^2}^2.
\end{aligned}
\end{equation}
On the other hand, still by the Fourier-Plancherel theorem, we have
\begin{align*}
\int_{\mathbb{T}^2} \Big|\sum_{|k|\geq n+1}\widehat{g}_k e^{2i\pi k\cdot x}\Big|^2h\,dx
&\leq \|h\|_{L^\infty}\sum_{|k|\geq n+1}|\widehat{g}_k|^2\leq 
\frac{ \|h\|_{L^\infty}}{n^2}\sum_{|k|\geq n+1}|k\widehat{g}_k|^2
\leq  \frac{\|h\|_{L^{\infty}}}{n^2} \| \nabla g\|_{L^2}^2.
\end{align*}
Choosing $n=\sqrt{1+\|h\|_{L^{\infty}}}$ gives \eqref{logg}.
\end{proof}

\bigbreak\noindent
\noindent \textbf{Acknowledgments.~} 
The authors are indebted to the anonymous referee for his fruitful suggestions that contributed to
improve the present article.
The first author (R. Danchin) is partially 
supported by \emph{Institut Universitaire de France}. The second author (L.-Y. Shou) is supported by National Natural Science Foundation of China (Grant No. 12301275).

\bibliographystyle{abbrv} 
\parskip=0pt
\small
\bibliography{Reference}

\vspace{3ex}

\noindent (R. Danchin)\par\nopagebreak
\noindent\textsc{Univ Gustave Eiffel, CNRS, LAMA UMR8050, Univ Paris Est Creteil, 94010 Creteil, France}

Email address: {\texttt{danchin@u-pec.fr}}

\vspace{3ex}

\noindent (L.-Y. Shou)\par\nopagebreak
\noindent\textsc{School of Mathematical Sciences, Ministry of Education Key Laboratory of NSLSCS, and Key Laboratory of Jiangsu Provincial Universities of FDMTA, Nanjing Normal University, Nanjing, 210023, P. R. China}

Email address: {\texttt{shoulingyun11@gmail.com}}

\end{document}